\documentclass[12pt]{article}
\usepackage{amsmath, amsfonts, amsthm, amssymb, verbatim}
\usepackage{graphicx}
\usepackage[mathcal]{euscript}
\usepackage{amstext}
\usepackage{amssymb,mathrsfs}
\usepackage{bbm}
\usepackage{xcolor}
\usepackage[utf8]{inputenc}
\usepackage{aeguill}
\newcommand{\R}{\mathbb R}
 \newcommand{\Z}{\mathbb Z}

\newcommand{\ve}{\varepsilon}

\newtheorem{theorem}{Theorem}[section]
\newtheorem{proposition}[theorem]{Proposition}
 \newtheorem{remark}[theorem]{Remark}

\newtheorem{corollary}[theorem]{Corollary}
\newtheorem{definition}[theorem]{Definition}

\begin{document}
\title{An Eulerian-Lagrangian Formulation 
of the Compressible 
Euler Equations with Vacuum}

\author{Wladimir Neves\footnote{Instituto de Matem\'atica, Universidade Federal do Rio de Janeiro, 
Brazil. E-mail:  {\sl wladimir@im.ufrj.br}
}, 
Christian Olivera\footnote{Departamento de Matem\'{a}tica, Universidade Estadual de Campinas, Brazil. 
E-mail:  {\sl  colivera@ime.unicamp.br}.
}}

\date{\today}

\maketitle

\noindent \textit{ {\bf Key words and phrases:} 
Compressible Euler equations, Eulerian-Lagrangian formulation,  
classical short-time solution, wellposedness.}

\vspace{0.3cm} \noindent {\bf MSC2010 subject classification:} 35Q31,
 35Q35, 76N10. 
 




\begin{abstract} 
In this paper, we present a novel Eulerian-Lagrangian formulation for the 
compressible isentropic Euler equations with vaccum. 
Using the developed Lagrangian flow map formulation, 
we show a short-time solution for a general pressure law. 
A particularly appealing feature of the approach used, it is well defined in the presence of vacuum, 
namely for compactly supported initial data which constitute an important problem in gas dynamics. 
Moreover, it does so without relying on any special symmetrization. 
While analogous results are well understood for incompressible fluids, 
the compressible setting, particularly in the presence of vacuum, remained open.
\end{abstract}



\section{Introduction\label{Intro}}

In this paper, we introduce an equivalence between the Eulerian and Lagrangian 
formulations of the compressible Euler equations for isentropic motion. By means of 
this formulation, we establish the well-posedness of regular solutions to the Euler 
equations with vacuum, see Definition \ref{RegSol}. We emphasize that the Lagrangian 
formulation is understood here in the sense of the flow map, in the geometric spirit 
introduced by Arnold \cite{Arnold} and later used by Ebin and Marsden \cite{Marsden}, 
rather than in Lagrangian coordinates (the so-called material description). The latter 
framework was employed, for instance, by Jang, Masmoudi \cite{Masmoudi} in their 
study of the compressible isentropic Euler equations with physical vacuum from the 
viewpoint of a free boundary problem. In this context, it is worth recalling that, 
the same framework developed in this paper for the compressible case, 
the Eulerian–Lagrangian formulation was previously studied for the incompressible 
Euler equations by Constantin \cite{Cont}, who established a local existence result
in H\"older Space $C^{1,\alpha}$, $(0<\alpha<1)$, for periodic solutions or satisfying suitable decay conditions. 
More recently, Pooley, Robinson \cite{Pooley} following \cite{Cont}, again in the incompressible case, 
considered this formulation 
in the Sobolev space $H^{s}(\mathbb{T}^d)$ with $s>\frac{d}{2} + 1$.

\medskip
Notably, the aim of this paper is to obtain a short-time solution to the 
compressible Euler system (see \eqref{EulerEqEntIsent} below)  
by applying our developed Lagrangian formulation. 
We emphasize that, in both the incompressible and compressible cases, 
the velocity vector field loses derivatives when analyzed from the Lagrangian flow map perspective. 
However, in the incompressible case, the incompressibility constraint enables the recovery of lost regularity, 
whereas in the compressible case, such recovery is not possible. 
This fundamental difference makes the compressible setting significantly more challenging. 
Indeed, in the incompressible framework, the evolution on the group of diffeomorphisms 
encodes the full physical structure of the problem, whereas this property breaks 
down in the compressible case. We were able to overcome these difficulties by applying a Nash-Moser iteration scheme 
to compensate for the loss of derivatives. 
Roughly speaking, following Nash \cite{Nash}, the key idea was to modify Newton’s scheme (see item 5 in the 
proof of Theorem \ref{maintheorem}) by introducing a smoothing operator at each iteration 
in order to compensate for the loss of regularity.
The reader is referred, for instance, to Moser \cite{Moser1, Moser2}, Hörmander \cite{Hormander1, Hormander2}, 
and Secchi \cite{Secchi}, 
where detailed presentations of the Nash–Moser iteration method can be found.

\medskip
A particularly appealing feature of the approach developed here is that, it is stable in the presence of vacuum, 
namely for compactly supported initial data which constitute an important problem in gas dynamics. 
Moreover, it does so without relying on any special symmetrization, such as the one introduced by 
Makino, Ukai, and Kawashima \cite{Makino}, which permits the density to vanish, albeit at the expense 
of obtaining a system that is no longer equivalent to the original Euler equations in the vacuum region. 
Conversely, the assumptions here on the initial data are dictated solely by minimal regularity requirements. 
In addition, the present framework covers general pressure laws depending on the density, 
without any restriction to a particular constitutive relation. 
In particular, the weighted trajectory mapping introduced in the proof of Theorem \ref{maintheorem}  
is the consistent treatment of the 
vacuum interface in polytropic gases with $1 < \gamma < 2$, where
the specific enthalpy possesses a gradient that becomes singular as the density approaches zero. 
In the vacuum region, the weight effectively deactivates the pressure term, allowing the momentum equation to recover the pure 
transport dynamics, see Remark \ref{RemVacumm}, thereby ensuring both 
mathematical stability and physical consistency across the fluid-vacuum interface.

\medskip
Finally, we recall that, in the absence of vacuum, that is to say, when the system is strictly hyperbolic throughout,
the theory of symmetric hyperbolic systems developed by Kato \cite{Kato}, and Lax \cite{Lax},  
yields the local in time existence of smooth solutions, see, for example, Majda \cite{Majda} and Benzoni-Gavage, Serre \cite{Serre}. 
Moreover, we observe that the formation of singularities in finite time for classical solutions was established by Sideris
\cite{Sideres}.

\subsection{Compressible Euler Equations} 

The basic balance laws of continuum physics in local form
are the evolution equations of conservation of mass, 
linear momentum and energy, respectively given by 
\begin{equation}
\label{GBCL}
\begin{aligned}
   &\partial_t \rho + {\rm div} (\rho \; \mathbf{u})= 0,
\\[5pt]
   &\partial_t (\rho \; \mathbf{u}) + {\rm div} (\rho \; \mathbf{u} \otimes \mathbf{u} )
   = {\rm div} \mathbf{T} + \rho \mathbf{b}
\\[5pt]
   &\partial_t (\rho E) + {\rm div} (\rho E  \mathbf{u} )
   = {\rm div} (\mathbf{T} \mathbf{u} - \mathbf{q}) + \rho r.
\end{aligned} 
\end{equation}
In this Eulerian formulation the scalar function 
$\rho(t,\mathbf{x})$ is the density, $\mathbf{u}(t,\mathbf{x})$ 
is the velocity vector field, 
$\mathbf{T}(t,\mathbf{x})$ is the symmetric stress tensor (e.g. non-polar materials),
$\mathbf{q}(t,\mathbf{x})$ is the heat flux vector, $\mathbf{b}(t,\mathbf{x})$ is the vector body forces, $e(t,\mathbf{x})$ is the internal energy, $r(t,\mathbf{x})$ is the energy supply
and $E$ denotes $e + \frac{|\mathbf{u}|^2}{2}$. 
Here $t \in (0,T)$ for some $T>0$, and $\mathbf{x} \in \Omega$, ($\Omega= \R^d$ or $\mathbb{T}^d$, where $\mathbb{T}^d= \R^d / 2 \pi \Z^d$), that is no boundary conditions are required, 
$(d= 2,3)$, and $\mathbf{u} \otimes \mathbf{u}= u_i u_j$, $(i,j= 1,2,3)$, where 
the usual summation convention is used. 

\medskip
We are interested in developing an Eulerian-Lagrangian formulation 
for the compressible Euler equations and applying it to show the 
existence of regular solutions for elastic fluids, 
inviscid (i.e. non-viscous), and non-heat conductive, which heat sources and also external 
forces are neglected, that is to say  
\begin{equation}
\label{INCF}
\begin{aligned}
\mathbf{T}&= - p \mathbf{I}, \qquad \mathbf{q}= 0, 
\\[5pt]
\mathbf{b}&= 0, \hspace{40pt} r= 0, 
\end{aligned}
\end{equation}
where $p$ is the pressure. 
Then, the system \eqref{GBCL} under assumptions \eqref{INCF}
turns into the so called (compressible) Euler equations
\begin{equation}
\label{EulerEq}
\begin{aligned}
   &\partial_t \rho +  \mathbf{u} \cdot \nabla \rho + \rho \; {\rm div} \mathbf{u}= 0,
\\[5pt]
   &\rho \, \big( \partial_t \mathbf{u} + (\mathbf{u} \cdot \nabla) \mathbf{u} \big) + \nabla p= 0, 
\\[5pt]
   &\rho \, \big( \partial_t e + \mathbf{u} \cdot \nabla e \big)
   + p \; {\rm div} \mathbf{u}= 0,
\end{aligned} 
\end{equation}
where we have used elementary manipulation (valid for smooth solutions),
and the last equation in \eqref{EulerEq} is usually called 
balance of internal energy.  

\medskip
One remarks that, the system \eqref{EulerEq}  has too 
few equations for too many unknowns, hence in order to close it suitable 
equations of state are needed. For instance, we may specify the pressure 
$p= p(\rho, \theta)$ and the internal energy $e= e(\rho, \theta)$,
where $\theta> 0$ is the thermodynamic temperature. Moreover, 
for an elastic fluid, the following Gibbs relation must be satisfied, that is, 
\begin{equation}
\label{GR}
 d(\rho \, s)= \frac{1}{\theta} \big(d(\rho \, e) - h \, d \rho \big),
\end{equation}
where $s$ is the (specific) entropy, and $h(\rho, \theta)$ is the enthalpy.  
{In addition, we impose the hypothesis of thermodynamic stability}
\begin{equation} 
\label{ThSt}
	\frac{\partial p(\rho, \theta)}{\partial \rho} > 0,\ \frac{\partial e(\rho, \theta)}{\partial \theta} > 0. 
\end{equation} 
Then, from \eqref{EulerEq} and \eqref{GR} we may write the Euler 
equations in a more convenient form
\begin{equation}
\label{EulerEqEnt}
\begin{aligned}
   &\partial_t \rho +  \mathbf{u} \cdot \nabla \rho + \rho \; {\rm div} \mathbf{u}= 0,
\\[5pt]
   &\rho \, \big( \partial_t \mathbf{u} + (\mathbf{u} \cdot \nabla) \mathbf{u}  + \nabla h - \theta \, \nabla s\big) = 0, 
\\[5pt]
   &\rho \, \theta \, \big( \partial_t s + \mathbf{u} \cdot \nabla s \big)= 0. 
\end{aligned} 
\end{equation}

\bigskip
The variables in the Euler equations can be described 
in terms of $t$ and the material 
point $\mathbf{a}$, namely Lagrangian formulation, see P. Constantin \cite{Cont}
(also \cite{Pooley}) in case of incompressible Euler equations. 
More precisely, the Lagrangian 
variables are path mappings $\mathbf{a} \mapsto \mathbf{X}_t(\mathbf{a})$, and the connection 
between the Eulerian description and the Lagrangian one is given when 
$\mathbf{X}_t$, called flow,
is the solution for each $t \in (0,T)$, of the following
system of differential equations
\begin{equation}
\label{CPODE}
   \left \{
   \begin{aligned}
   &\frac{d}{dt} \mathbf{X}_t(\mathbf{a})= \mathbf{u}(t, \mathbf{x}), \quad \mathbf{x}= \mathbf{X}_t(\mathbf{a}), 
   \\[5pt]
   & \mathbf{X}_0(\mathbf{a})= \mathbf{a}.
   \end{aligned}
      \right.
\end{equation}
One recalls that, when $\mathbf{x}= \mathbf{X}_t(\mathbf{a})$ 
we have for $\phi$ be a scalar field 
$$
   D_t \phi(t, \mathbf{X}_t(\mathbf{a}))= \big(\partial_t \phi \big) (t, \mathbf{X}_t(\mathbf{a})) 
   + \big(\nabla_{\!\mathbf{x}} \phi \big) (t, \mathbf{X}_t(\mathbf{a}))  \cdot \mathbf{u}(t, \mathbf{X}_t(\mathbf{a})), 
$$
and $\mathbf{w}$ be a vector field 
$$
   D_t \mathbf{w}(t, \mathbf{X}_t(\mathbf{a}))= \big(\partial_t \mathbf{w}\big) (t, \mathbf{X}_t(\mathbf{a})) 
   + \big(\nabla_{\!\mathbf{x}} \mathbf{w}\big) (t, \mathbf{X}_t(\mathbf{a}))  \, \mathbf{u}(t, \mathbf{X}_t(\mathbf{a})). 
$$
Similarly, it follows that 
$$
\begin{aligned}
   \nabla_{\! \mathbf{a}} \phi(t, \mathbf{X}_t(\mathbf{a}))&
   =  \big(\nabla_{\! \mathbf{a}} \mathbf{X}_t(\mathbf{a})\big)^\ast \big(\nabla \phi \big)(t, \mathbf{X}_t(\mathbf{a})), 
   \\[5pt]
   \nabla_{\! \mathbf{a}} \mathbf{w}(t, \mathbf{X}_t(\mathbf{a}))&
   =  \big(\nabla \mathbf{w}\big)(t, \mathbf{X}_t(\mathbf{a})) \, \nabla_{\! \mathbf{a}} \mathbf{X}_t(\mathbf{a}).
\end{aligned}
$$
We observe that, $M^\ast$ denotes the adjoint of a matrix $M$,  
and $D_t f$ is usually called material time derivative of $f$, 
which is the time derivative of $f$ following the path 
of the material point. 
Moreover, it is not difficult to show that 
$$
\begin{aligned}
D_t (f \, g)&=  D_t f \; g + f \; D_t g,
\\[5pt]
D_t(\nabla f)&= \nabla D_t f - (\nabla \mathbf{u})^\ast \nabla f. 
\end{aligned}
$$
Therefore, the Euler equations can be written as (since $\theta> 0$)
\begin{eqnarray}
\label{BMASS}
   D_t \rho +  \rho  \, {\rm div} \mathbf{u}&=& 0,
   \\[5pt]
\label{BLM}
 \rho \, \big(D_t \mathbf{u} + \nabla h - \theta \nabla s \big)&=& 0, 
   \\[5pt]
\label{BLH}
\rho \, \big(D_t s\big)&=& 0. 
\end{eqnarray}

Let us analyse equation \eqref{BLM}. If $\rho$ is uniformly positive, one can divide by $\rho$ to obtain
\begin{equation}
\label{BLM1}
   D_t \mathbf{u} + \nabla h - \theta \nabla s= 0.
\end{equation}
However, when $\rho = 0$, the equation \eqref{BLM} degenerates, that is to say,  the linear momentum equation becomes simply $0 = 0$, hence at such vacuum points one could, in principle, have either
\begin{equation*}
  D_t \mathbf{u} + \nabla h - \theta \nabla s = 0, \quad \text{or} \quad D_t \mathbf{u} + \nabla h - \theta \nabla s \neq 0.
\end{equation*}
Formally both satisfy equation \eqref{BLM}, but only the former is compatible with smootheness assumption.
Indeed, the product
\begin{equation*}
\big(\rho \big) \,   \big(D_t \mathbf{u} + \nabla h - \theta \nabla s \big)
\end{equation*}
is a smooth function vanishing identically. If the first factor ($\rho$) vanishes on a nontrivial (open) set, then
smoothness forces the other factor to vanish there as well. Otherwise, the product would fail to 
be differentiable across the vacuum boundary.
Thus, for regular solutions we must have the, usually called, compatibility condition
$$
 D_t \mathbf{u} + \nabla h - \theta \nabla s = 0, \quad \text{where $\rho= 0$}.
$$

\medskip
An analogous argument applies to equation \eqref{BLH}, which 
expresses the fact that the changes of states are adiabatic
(absence of heat conducting between different 
parts of the fluid),
that is, the entropy is constant following each fluid particle (Lagrangian formulation) but not 
necessarily through the fluid (Eulerian formulation).  
Although, if (as usually happens) the entropy is constant throughout the volume of the 
fluid at some initial instant, then it remais the same constant value everywhere at each 
time and any subsequence motion of the fluid. Therefore, we may assume that the entropy 
$s= const$, and we do so in what follows, which means that the motion is 
isentropic. On the other hand, if a solution of the full Euler equations
experiences a discontinuity, 
this assumption cannot be assumed. 

\medskip
Consequently,  the Euler equations reduce to consider the conservation of mass, 
and the linear momentum equations, where the enthalpy depends just on the density 
$\rho$ (notice $h(0)= 0$). 
Then, we obtain the following system
\begin{equation}
\label{EulerEqEntIsent}
\left\{
\begin{aligned}
   &\partial_t \rho + (\mathbf{u} \cdot \nabla) \rho + \rho \; {\rm div} \mathbf{u}= 0,
\\[5pt]
   & \partial_t \mathbf{u} + (\mathbf{u} \cdot \nabla) \mathbf{u}  + \nabla h(\rho)= 0, 
\end{aligned} 
\right .
\end{equation}
which should be supplemented with the initial data $(\rho_0, \mathbf{u}_0)$ given in the appropriate spaces. 
    
\begin{remark}
\label{RemVacumm}
1. Physically, the isentropic Euler equations should reduce to the pressureless transport equation in vacuum regions $(\rho = 0)$, 
as the absence of mass precludes any pressure-driven acceleration, that is, we should have the purely inertial dynamics
\begin{equation}
\label{EulerEqEntIsentVacuum}
    \partial_t \mathbf{u} + (\mathbf{u} \cdot \nabla) \mathbf{u}= 0, \quad \text{where $\rho= 0$}.
\end{equation}

2. Mathematically, however, the classical formulation \eqref{EulerEqEntIsent} relies on the regularity of 
the enthalpy $h(\rho)$ at the origin, which fails for the polytropic case $1 < \gamma < 2$ since $h'(\rho)$ diverges as $\rho \to 0$.
Our construction of the weight $\mathcal{W}^\varepsilon$ $($see the proof of Theorem \ref{maintheorem} $)$ 
effectively decouples the momentum balance from the singular behavior of the enthalpy in the low-density limit. 
By ensuring that $\mathcal{W}^\varepsilon \to 0$ faster than any algebraic blow-up of $\nabla h$, we recover  
\eqref{EulerEqEntIsentVacuum}, making the existence of regular solutions independent of the specific singularity of the equation of state.
\end{remark}  

\section {Notation and background}

In this section, we fix the notations and 
gather the preliminary results necessary 
to show the short-time solution to the compressible Euler system  \eqref{EulerEqEntIsent}
using the Lagrangian formulation developed in this paper.

\medskip
First, let $U \subset \R^d$ be an open set. 
We denote by $dx, d\xi$, etc. 
the Lebesgue measure on $U$ and by $L^p(U)$, $p \in [1,+\infty)$, the set of (real or complex) $p$-summable functions 
with respect to the Lebesgue measure, also $L^\infty(U)$ is 
the set of measurable functions in which its absolute value has the essential supremum finite,
(vector ones should be understood 
componentwise).
As usual, the Schwartz space is denoted by $\mathcal{S}(\R^d)$, and
$\mathcal{S}^{\prime}(\R^d)$ denotes the set of tempered distributions.
We denote by
$\mathcal{F} \varphi(\xi) \equiv \widehat{\varphi}(\xi)$
the Fourier Transform of $\varphi$, which is an
isometry in $L^2(\R^d)$.

\medskip
Observation confirms that the Sobolev spaces $H^s$ (energy of $L^2-$type) represent the convenient and effective class 
of functions for seeking solutions to quasilinear (symmetric) system of conservation laws, which is the 
case of compressible (isentropic) Euler equations when written in the Eulerian formulation as in \eqref{EulerEqEntIsent}.
Although, it seems more interesting to apply the Lagrangian formulation developed 
in the setting of the Bessel space $H_{p}^{\beta}(\Omega)$, which we describe below.
First, we define $H_{p}^{\beta}(\R^d)$, (or simply written $H_{p}^{\beta}$)
via Fourier Transform. Precisely, we have
\begin{equation}
\label{HsDEF}
    H^\beta_p(\R^d):= \Big\{ u \in \mathcal{S}^{\prime}(\R^d) : (1+|\xi|^{2})^{\beta/2} \ \widehat{u}(\xi) \in L^p(\R^d) \Big\}
\end{equation}
and we observe that the above definition is valid for any $p \in (1,\infty)$, and $\beta \in \R$. Moreover, 
$H^\beta_2= H^\beta$ is a Hilbert space with the scalar product
$$
   (u,v)_{H^{\beta}} = \int_{\R^{d}} (1+|\xi|^{2})^{\beta} \ \widehat{u}(\xi) \ \overline{\widehat{v}}(\xi) \ d\xi,
$$
and we may introduce a norm in $H^\beta_p$ by means of
$$
   \|u\|_{H^\beta_p}= \| (I - \Delta)^{\frac{\beta}{2}} u \|_{L^p}.  
$$
Vector fields $\mathbf{u}$ in $H^\beta_p$ should be understood 
componentwise. 

\medskip
Similarly, we may consider the Bessel space on torus. 
Let $D(\mathbb{T}^{d})$ be the collection of all infinitely differentiable functions on
$\mathbb{T}^{d}$. Then $D^{\prime}(\mathbb{T}^{d})$ stands for the topological dual of
$D(\mathbb{T}^{d})$. We denote the Fourier coefficients of $u\in D^{\prime}(\mathbb{T}^{d})$ 
by $\hat{u}(k):= \frac{1}{(2\pi)^{d/2}} u(e^{2i\pi \langle k,\cdot\rangle})$. 
As there should be no risk of confusion, we will use the same notations for the Bessel norm and Bessel operator on the torus as 
in the whole space. Hence for any $\beta\in \mathbb{R}$, we define the Bessel potential operator $(I-\Delta)^{\beta/2}$ applied to $u\in D'(\mathbb{T}^{d})$ by
\begin{equation*}
(I-\Delta)^{\frac{\beta}{2}} u(x)  =  \frac{1}{(2\pi)^{d/2}} \sum_{k\in \mathbb{Z}} (1+|k|^2)^{\frac{\beta}{2}} \, \hat{u}(k)\, e^{-2i\pi \langle k,x\rangle},
\end{equation*}
and denote its norm by
\begin{align}\label{eq:Besselnorm}
\|u\|_{\beta,p} =  \left\| (I-\Delta)^{\frac{\beta}{2}} u\right\|_{L^p(\mathbb{T}^d)}.
\end{align}
As in \cite{SchmeisserTriebel} , $H_{p}^{\beta}(\mathbb{T}^{d})$ is
defined for $p\in\left(1,\infty\right)$ and $\beta\in \mathbb R$ by
\begin{align*}
H_{p}^{\beta}(\mathbb{T}^{d}):= \Big\{ u \in D^{\prime}(\mathbb T^d) ;\, \|u\|_{\beta,p}<\infty \Big\}.
\end{align*}

\section{Eulerian-Lagrangian Equivalent Formulations}

To begin the study of the Eulerian-Lagrangian equivalent formulations for compressible Euler equations, let us consider 
companion to system \eqref{CPODE} the so-called the back-to-labels map $\mathbf{A}$, defined by
$$
   \mathbf{A}_t(\cdot):= \mathbf{X}_t^{-1}(\cdot), \quad \text{that is} \quad a^i= A^i_t(\mathbf{X}_t(\mathbf{a})),
$$
which must satisfy the active vector equation, more precisely, given the vector field $\mathbf{u}$ the back-to-labels map is the solution of the
following Cauchy problem
\begin{equation}
\label{back-to-labels}
\begin{aligned}
D_t \mathbf{A}_t(\mathbf{x})\equiv  \partial_t \mathbf{A}_t(\mathbf{x}) + (\mathbf{u} \cdot \nabla) \mathbf{A}_t(\mathbf{x})&= 0, 
\\
\mathbf{A}_0(\mathbf{x})&= \mathbf{x}. 
\end{aligned}
\end{equation}
Then, we introduce the virtual velocity field $\mathbf{v}$ obtained by the solution of the Cauchy problem 
\begin{equation}
\label{virtualvelocity}
\begin{aligned}
D_t \mathbf{v}(t,\mathbf{x})\equiv \partial_t \mathbf{v}(t,\mathbf{x}) + (\mathbf{u} \cdot \nabla) \mathbf{v}(t,\mathbf{x})&= 0,
\\
\mathbf{v}(0,\mathbf{x})&= \mathbf{u}_0(\mathbf{x}), 
\end{aligned}
\end{equation}
which solution is given by $\mathbf{v}(t,\mathbf{x})= \mathbf{u}_0(\mathbf{A}_t(\mathbf{x}))$.

\medskip
To follow, we observe that the first equation in \eqref{EulerEqEntIsent} corresponds to the continuity equation written in Eulerian coordinates, 
which could be also written in Lagrangian coordinates as 
\begin{equation}
\label{DENSLAG}
\begin{aligned}
\rho(t,\mathbf{x})&= J \mathbf{A}_t(\mathbf{x}) \, \rho_0(\mathbf{A}_t(\mathbf{x}) ), 
\\[5pt]
\rho(0,\mathbf{x})&= \rho_0(\mathbf{x}), 
\end{aligned}
\end{equation}
where $J \mathbf{A}_t(\mathbf{x})$ is the Jacobian of the back-to-labels map. Similarly, we may write the continuity equation in integral form, that is, 
\begin{equation}
\label{DENSLAGINT}
\rho(t,\mathbf{x})= \rho_0(\mathbf{A}_t(\mathbf{x}) ) - \int_0^t (\rho \, {\rm div} \mathbf{u})(s,\mathbf{A}_{t-s}(\mathbf{x})) \ ds. 
\end{equation}

\medskip
Finally, the formulation comprises the following Cauchy problem
\begin{equation}
\label{calculus}
\begin{aligned}
D_t q(t,\mathbf{x}) \equiv  \partial_t q(t,\mathbf{x}) + (\mathbf{u} \cdot \nabla) q(t,\mathbf{x})&= H(t,\mathbf{x}) - K(t,\mathbf{x}),
\\
q(0,x)&= 0,
\end{aligned}
\end{equation}
where we write, conveniently, $H(t,\mathbf{x})= h(\rho(t,\mathbf{x}))$ and $K(t,\mathbf{x})= \frac{1}{2} |\mathbf{u}(t,\mathbf{x})|^2$. 

\bigskip
Since the continuity equation has already been formulated in both the Eulerian and Lagrangian descriptions, 
it remains to study the linear momentum equation, which is done in the following two sections. 

\subsection {From Eulerian to Lagrangian}

Let $\mathbf{u}(t,\mathbf{x})$ 
be the velocity vector field, due to \eqref{CPODE} we may write  
$$\tilde{\mathbf{u}}(t,\mathbf{a})= \mathbf{u}(t,\mathbf{X}_t(\mathbf{a}))$$ 
called the Lagrangian velocity of a fluid element, which is given for each component by 
$$
   \tilde{u}^i(t, \mathbf{a})= \frac{\partial X^i}{\partial t} (t, \mathbf{a}). 
$$
Hence we have
$$
\begin{aligned}
  \frac{\partial \tilde{u}^i}{\partial t}(t, \mathbf{a})\equiv \frac{\partial^2 \! X^i}{\partial t^2} (t, \mathbf{a})&= - \frac{\partial}{\partial x^i}\big(H(t, \mathbf{X}_t(\mathbf{a}))\big)
  \\[5pt] 
  &= - \frac{\partial A^j}{\partial x^i} \frac{\partial \tilde{H}}{\partial a^j}(t, \mathbf{a}), 
 \end{aligned}  
$$
where $\tilde{H}(t,\mathbf{a})= H(t, \mathbf{X}_t(\mathbf{a}))$, and we have used the second equation in \eqref{EulerEqEntIsent}. 
Premultiplying by $(\nabla \mathbf{X}_t)^\ast$, we obtain 
\begin{equation}
\label{Eq10}
\begin{aligned}
  \frac{\partial^2 \! X^i}{\partial t^2} (t, \mathbf{a})  \frac{\partial \! X^i}{\partial a^j} (t, \mathbf{a})= 
  - \frac{\partial \tilde{H}}{\partial a^j}(t, \mathbf{a}).
 \end{aligned}  
\end{equation}

\smallskip
Now, a straightforward calculation shows that
\begin{equation}
\label{Eq20}
\begin{aligned}
    \frac{\partial^2 \! X^i}{\partial t^2} (t, \mathbf{a})  \frac{\partial \! X^i}{\partial a^j} (t, \mathbf{a})= 
    \frac{\partial}{\partial t} \Big( \frac{\partial X^i_t}{\partial t}   \frac{\partial X^i_t}{\partial a^j}  \Big) (t, \mathbf{a})
- \frac{1}{2}  \frac{\partial}{\partial a^j} \Big|  \frac{\partial X^i}{\partial t} (t, \mathbf{a}) \Big|^2.
\end{aligned}
\end{equation}
Then, it follows from \eqref{Eq10}, \eqref{Eq20} that 
$$
\begin{aligned}
   \frac{\partial}{\partial t} \Big( \frac{\partial X^i_t}{\partial t}(t, \mathbf{a})   \frac{\partial X^i_t}{\partial a^j}(t, \mathbf{a})  \Big)
   &= - \frac{\partial}{\partial a^j} \Big( \tilde{H}(t, \mathbf{a}) - \tilde{K}(t, \mathbf{a}) \Big), 
\end{aligned}
$$
with obvious notation. Integrating from $0$ to $t$ and premultiplying by $(\nabla \mathbf{A}_t)^\ast$, we have 
$$
\begin{aligned}
    \tilde{u}^i(t, \mathbf{a})&= u_0^j(\mathbf{a})  \, \frac{\partial A^j_t}{\partial x^i} 
   - \int_0^t  \frac{\partial A^j_t}{\partial x^i} \frac{\partial}{\partial a^j} \Big( \tilde{H}(s, \mathbf{a}) - \tilde{K}(s, \mathbf{a}) \Big) \, ds,
\end{aligned}
$$
which can be written as 
$$
\begin{aligned}
    u^i(t,\mathbf{X}_t(\mathbf{a}))&= u_0^j(\mathbf{a})  \, \frac{\partial A^j_t}{\partial x^i} 
   - \int_0^t  \frac{\partial A^j_{t-s}}{\partial x^i} \frac{\partial}{\partial x^j} \Big( H(s,\mathbf{X}_s(\mathbf{a})) - K(s,\mathbf{X}_s(\mathbf{a})) \Big) \, ds.
\end{aligned}
$$
Finally, making $\mathbf{a}= \mathbf{A}_t(\mathbf{x})$ we obtain 
$$
\begin{aligned}
\mathbf{u}(t, \mathbf{x})&= (\nabla \! \mathbf{A}_t)^\ast \mathbf{u}_0(\mathbf{A}_t(\mathbf{x})) 
- \int_0^t  (\nabla \! \mathbf{A}_{t-s})^\ast \,  (\nabla H)(s,\mathbf{A}_{t-s}(\mathbf{x})) \ ds 
\\[5pt]
&+ \int_0^t  (\nabla \! \mathbf{A}_{t-s})^\ast \,  (\nabla K)(s,\mathbf{A}_{t-s}(\mathbf{x})) \ ds
\\[5pt]
&= (\nabla \! \mathbf{A}_t)^\ast \mathbf{u}_0(\mathbf{A}_t(\mathbf{x})) 
\\[5pt]
&-  \int_0^t \nabla H(s,\mathbf{A}_{t-s}(\mathbf{x})) \ ds + \int_0^t  \nabla K(s,\mathbf{A}_{t-s}(\mathbf{x})) \ ds, 
\end{aligned} 
$$
which is to say 
\begin{equation}
\label{Velocity}
\mathbf{u}= (\nabla \! \mathbf{A}_t)^\ast \mathbf{v} - \nabla q. 
\end{equation}

Therefore, we have used the Eulerian description of the velocity vector field given by $(\ref{EulerEqEntIsent})_2$ 
and obtained the Lagrangian description of the velocity given by \eqref{Velocity}. Equations \eqref{back-to-labels}--\eqref{calculus}
are used implicitly. 

\subsection {From Lagrangian to Eulerian}
 
 Now, we take the velocity vector field given by the Lagrangian description, that is equation \eqref{Velocity}. 
 Then computing the material derivative of it , and using equations \eqref{back-to-labels}--\eqref{calculus}, 
 we have
 $$
 \begin{aligned}
D_t \mathbf{u}&= D_t \big( (\nabla \! \mathbf{A}_t)^\ast \mathbf{v} \big) - D_t\big( \nabla q \big)
   \\[5pt]
   &= D_t \big( (\nabla \! \mathbf{A}_t)^\ast \big) \mathbf{v} + (\nabla \! \mathbf{A}_t)^\ast D_t \mathbf{v} 
   - \nabla \big(D_t q \big) + (\nabla \mathbf{u})^\ast \nabla q
  \\[5pt]
   &= (\nabla \! D_t \mathbf{A}_t)^\ast \mathbf{v} - (\nabla \mathbf{u})^\ast (\nabla \mathbf{A}_t)^\ast \mathbf{v} 
   - \nabla \big(D_t q \big) + (\nabla \mathbf{u})^\ast \nabla q
    \\[5pt]
   &= - (\nabla \mathbf{u})^\ast \mathbf{u} - (\nabla \mathbf{u})^\ast \nabla q 
   - \nabla \big(D_t q \big) + (\nabla \mathbf{u})^\ast \nabla q
    \\[5pt]
   &= - \nabla K - \nabla H + \nabla K= - \nabla h(\rho), 
   \end{aligned}
 $$
 hence we obtain the second equation in \eqref{EulerEqEntIsent}.

\subsection {Lagrangian Flow Map Framework}
\label{Lagrangian formulation}

The previous calculations lead to the conclusion that, given
$(\rho_0, \mathbf{u}_0)$ all the equations \eqref{back-to-labels}--\eqref{calculus}
depends on $\rho, \mathbf{u}$, which are given, respectvely, by equations \eqref{DENSLAGINT} and \eqref{Velocity}. Therefore, 
we have proved the following important
\begin{theorem}
\label{EqFormulations}
A pair of functions $(\rho, \mathbf{u})$ satisfies the system
\eqref{EulerEqEntIsent} with initial data $(\rho_0, \mathbf{u}_0)$ if, and only if, the equations  \eqref{DENSLAGINT}, \eqref{Velocity} are satisfied,
more precisely
$$
\begin{aligned}
&\rho(t,\mathbf{x}) - \rho_0(\mathbf{A}_t(\mathbf{x}) ) + \int_0^t (\rho \, {\rm div} \mathbf{u})(s,\mathbf{A}_{t-s}(\mathbf{x})) \ ds= 0,
\\[5pt]
&\mathbf{u}(t, \mathbf{x}) -(\nabla \! \mathbf{A}_t)^\ast \mathbf{u}_0(\mathbf{A}_t(\mathbf{x})) 
+\!\! \int_0^t \!\! \nabla (h( \rho(s, \mathbf{A}_{t-s}(\mathbf{x}))) - \frac{|\mathbf{u}(s,\mathbf{A}_{t-s}(\mathbf{x}))|^2}{2}) ds= 0, 
\end{aligned}
$$
where the back-to-labels map $\mathbf{A}$ is given by \eqref{back-to-labels}.
\end{theorem}

This is the new framework we use in this paper to show the existence of local regular solutions to compressible 
Euler Equations.

\section{Regular Solutions: Support Invariance}

This section beggins with the definition of regular solutions, 
that is, following Makino, Ukai, and Kawashima \cite{Makino}, we present 
what is meant to be a local smooth solution of the isentropic 
compressible Euler equations with vacuum.  

\begin{definition}[Regular Solution]
\label{RegSol}
Let $\beta>\frac{d}{p}+1$ with $p \in (1,\infty)$ be a real number, 
$h \in C^2(0,\infty)$, and consider the initial conditions 
$\rho_0 \in H_p^{\beta}(\Omega)$ of compact support, and $\mathbf{u}_0 \in H_p^{\beta}(\Omega)$.
Let $T> 0$ be a finite constant, hence a pair of functions
$$
  (\rho, \mathbf{u}) \in \Big(C([0,T]; H_p^\beta(\Omega)) \cap C^1([0,T]; H_p^{\beta-1}(\Omega)) \Big)^2
$$
satisfying \eqref{EulerEqEntIsent} where $\{\rho> 0\}$  $($respectively \eqref{EulerEqEntIsentVacuum} where $\{\rho= 0\}$$)$ with initial data 
$(\rho_0, \mathbf{u}_0)$ is called a regular solution.  
\end{definition}

\begin{remark}
In the previous definition, if $\Omega= \mathbb T^d$, then we assume that ${\rm spt} \rho_0\neq \mathbb T^d$.
Moreover, compare it with the one on page 253 of \cite{Makino}. 
In particular, we note item (ii) in that definition, which is a technical condition arising from the symmetrization employed there.
\end{remark}

\begin{remark}
The regularity class specified in Definition \ref{RegSol} ensures that the solutions are classical in the sense that they 
satisfy the equations point-wise. Indeed, since we assume $\beta > \frac{d}{p} + 1$, the Sobolev Embedding Theorem for the Bessel Spaces $H_p^\beta(\Omega)$ implies
$$
    H_p^\beta(\Omega) \hookrightarrow C^{1, \alpha}(\Omega), \quad \text{where } \alpha = \beta - 1 - \frac{d}{p} > 0.
$$
Consequently, the spatial derivatives $\nabla \rho$ and $\nabla \mathbf{u}$ are continuous. Furthermore, 
the condition $(\rho, \mathbf{u}) \in C^1([0,T]; H_p^{\beta-1}(\Omega))$ combined with the embedding $H_p^{\beta-1}(\Omega) \hookrightarrow C^0(\Omega)$ 
guarantees that the time derivatives $\partial_t \rho$ and $\partial_t \mathbf{u}$ are also continuous. Thus, we have the following inclusion
\begin{equation}
    \left( C([0,T]; H_p^\beta(\Omega)) \cap C^1([0,T]; H_p^{\beta-1}(\Omega)) \right)^2 \subset C^1([0,T] \times \Omega)^2.
\end{equation}
The term regular solution is preferred over classical solution in this context because it emphasizes that the 
solution persists in the same functional space as the initial data. While a classical solution only requires 
the continuity of derivatives, a regular solution maintains the precise $H_p^\beta$ regularity, 
which is essential for the stability estimates and the local Lipschitz continuity of the data-to-solution map.
\end{remark}

Then, we have the following 

\begin{proposition}[Invariance of Support]
\label{Invariance}
Let $(\rho, \mathbf{u})$ be a regular solution. 
Then, for each $t \in [0,T]$, the support of the solution is transported by the flow $\mathbf{X}_t$, such that, 
\[
{\rm spt}(\rho(t), \mathbf{u}(t)) = \mathbf{X}_t({\rm spt}(\rho_0, \mathbf{u_0})).
\]
In particular, the topological properties of the support are invariant.
\end{proposition}

\begin{proof}
1. Since $\beta>\frac{d}{p}+1$ with $p \in (1,\infty)$, it follows from the Sobolev Embedding Theorem that, 
$H_p^{\beta}(\Omega) \hookrightarrow C^1_0(\Omega)$. Therefore, since $\mathbf{u} \in C([0,T]; H_p^\beta(\Omega))$,
there exists a constant $M> 0$, such that, for each $t \in [0,T]$
$$
    \|\nabla \mathbf{u}(t,\cdot)\|_{L^\infty(\Omega)}\le M.
$$
Let $\mathbf{X}_t(\mathbf{a})$ be given by \eqref{CPODE}, 
and $J\mathbf{X}_t(\mathbf{a})$ the associated Jacobian which satisfies the Liouville equation, that is to say,  
\begin{equation}
\label{LiouvilleEq}
   \left \{
   \begin{aligned}
   &\frac{d}{dt} J \mathbf{X}_t(\mathbf{a})= {\rm div} \mathbf{u}(t, \mathbf{X}_t(\mathbf{a})) \, J \mathbf{X}_t(\mathbf{a}),
   \\[5pt]
   & J\mathbf{X}_0(\mathbf{a})= 1.
   \end{aligned}
      \right.
\end{equation}
Therefore, we obtain 
\[
   J\mathbf{X}_t(\mathbf{a})=\exp\!\left(\int_0^t {\rm div} \mathbf{u}(s, \mathbf{X}_s(\mathbf{a})) \big)\,ds\right),
\]
and hence we have for all $(t,\mathbf{a}) \in [0,T] \times \Omega$, 
\begin{equation}
\label{LimUnifJac}
e^{-Mt} \le J\mathbf{X}_t(\mathbf{a}) \le e^{Mt}. 
\end{equation}

\medskip
Now, we recall the continuity equation written in Lagrangian coordinates, that is, 
$$
  \rho\big(t,\mathbf{X}_t(\mathbf{a})\big) \, J\mathbf{X}_t(\mathbf{a})= \rho_0(\mathbf{a})
$$
from which and \eqref{LimUnifJac}, it follows that: 
\begin{itemize}
\item[i)] If $\rho_0(\mathbf{a})= 0$, then $\rho(t,\mathbf{X}_t(\mathbf{a}))= 0$ for all $t \in [0,T]$, that is, the vacuum is transported.

\item[ii)] If $\rho_0(\mathbf{a})> 0$, then for each $t \in [0,T]$, 
\[
\rho\big(t,\mathbf{X}_t(\mathbf{a})\big)= \frac{\rho_0(\mathbf{a})}{J\mathbf{X}_t(\mathbf{a})}\ \ge\ \rho_0(\mathbf{a})\,e^{-Mt}\ > 0,
\]
that is, the gas does not become vacuum in finite time.
\end{itemize}
Moreover, since $\mathbf{X}_t(\cdot)$ is a diffeomorphism of $\Omega$ (of class $C^1$ with strictly positive Jacobian), we have 
\begin{equation}
\label{SuppRho}
\{\rho(\cdot,t)>0\}= \mathbf{X}_t\big(\{\rho_0>0\}\big).
\end{equation}

\medskip
2. Finally, from the definition of regular solution and \eqref{SuppRho}, it follows that 
\begin{equation}
\label{UniqSol}
   \partial_t \mathbf{u} + (\mathbf{u} \cdot \nabla) \mathbf{u}= 0 \quad \text{as} \quad \rho_0= 0, 
\end{equation}
this concludes the proof
\end{proof}

\begin{corollary}
Under conditions of Definition \ref{RegSol}, equation \eqref{UniqSol} guarantees uniqueness of the regular solution, 
and ensures that the velocity vector field $\mathbf{u}$ is well defined in vaccuum regions. 
\end{corollary}

\section{Application: Existence of regular solution}
\label{mainexistence}

Next, we present the well-posedness theorem for regular solutions.  

\begin{theorem}
\label{maintheorem}
Let $\beta>\frac{d}{p}+1$ with $p \in (1,\infty)$ be a real number, 
$h \in C^2(0,\infty)$, and consider the initial conditions 
$\rho_0 \in H_p^{\beta}(\Omega)$ of compact support, $($in case $\Omega= \mathbb T^d$, ${\rm spt} \rho_0\neq \mathbb T^d)$, and $\mathbf{u}_0 \in H_p^{\beta}(\Omega)$.
Then, there exists a regular (short-time) solution in the sense of Definition \ref{RegSol}. 
Moreover, the solution depends continuously on the initial data; more precisely,
the data-to-solution map
\[
(\rho_0,\mathbf u_0)\ \longmapsto\ (\rho,\mathbf u)\in C([0,T];H_p^\beta(\Omega))^2
\]
is locally Lipschitz from $H_p^\beta(\Omega)\times H_p^\beta(\Omega)$ into $C([0,T];H_p^\beta(\Omega))^2.$
\end{theorem}

\begin{proof}
1. First, we fix 
$
\beta>\frac{d}{p}+1,
$
so that $H_p^\beta(\Omega)$ is an algebra and embeds into 
$C^{1,\alpha}$, with $\alpha=\beta-1-\frac{d}{p}> 0$.
Then, we define the scale of Bessel spaces on trajectories, that is to say,
\[
   \mathcal{X}_m := C\big([0,T];H_p^{\beta+m}(\Omega)\big)\times
                    C\big([0,T];H_p^{\beta+m}(\Omega)\big),
   \qquad (m\in\mathbb{N}).
\]
Given a trajectory $\mathbf{U}=(\rho,\mathbf{u})\in \mathcal{X}_m$, we define the following norm 
\[
   \|(\rho,\mathbf{u})\|_{\mathcal{X}_m}
   := \sup_{t\in[0,T]}\Big(\|\rho(t)\|_{H_p^{\beta+m}} + \|\mathbf{u}(t)\|_{H_p^{\beta+m}}\Big), 
\]
and similarly we consider 
\[
   \mathcal{Y}_m := C\big([0,T];H_p^{\beta+m-1}(\Omega)\big)\times
                    C\big([0,T];H_p^{\beta+m-1}(\Omega)\big). 
\]
Moreover, we consider 
$$
\mathcal{X}_\infty:=\bigcap_{m\ge0}\mathcal{X}_m, 
     \,\, \mathcal{Y}_\infty:=\bigcap_{m\ge0}\mathcal{Y}_m. 
$$

Now, for any $\rho_0 \in H_p^{\beta}(\Omega)$ with compact support and $\mathbf{u}_0 \in H_p^{\beta}(\Omega)$, fixed initial data, 
let us consider the trajectory mapping. 
Since the enthalpy is not assumed, necessarily, smooth up to the vacuum state, 
we introduce, for $\ve \in (0,1]$ fixed, a suitable weight function, defined by
$$
\mathcal{W}^\varepsilon_{t}(\mathbf{x}):= g_\varepsilon(\rho_0(\mathbf{A}_t(\mathbf{x}))),\qquad
g_\varepsilon(r)=
\begin{cases}
e^{-\varepsilon/r}, & r>0,\\
0, & r= 0,
\end{cases}
$$
and thus we have 
$\mathbf{F}_{\rho_0,\mathbf{u}_0} \equiv \mathbf{F}= (F_{cont}, F_{mom}): \mathcal{X}_0 \to \mathcal{Y}_0$, 
given by 
\begin{align*}
&F_{cont}(\rho,\mathbf{u})(t,\mathbf{x})
 = \rho(t,\mathbf{x})
   - \rho_0(\mathbf{A}_t(\mathbf{x})) 
   + \int_0^t (\rho\,{\rm div}\,\mathbf{u})\big(s,\mathbf{A}_{t-s}(\mathbf{x})\big)\, ds,
\\[7pt]
&F_{mom}(\rho,\mathbf{u})(t,\mathbf{x})
= \mathbf{u}(t,\mathbf{x})
   -(\nabla \! \mathbf{A}_t(\mathbf{x}))^\ast \mathbf{u}_0(\mathbf{A}_t(\mathbf{x})) 
\\[5pt]
& \hspace{40pt}
   + \int_0^t \Big( \mathcal{W}^\ve_{t-s}(\mathbf{x}) \, 
     \nabla \big(h\big( \rho\big(s,\mathbf{A}_{t-s}(\mathbf{x})\big)\big) \big)
     - \nabla \big(\tfrac12 
        \big|\mathbf{u}\big(s,\mathbf{A}_{t-s}(\mathbf{x})\big)\big|^2
     \big) \Big)\, ds, 
\end{align*}
where for simplicity we have 
omitted the superscript $\ve$ in $\rho^\ve,\mathbf{u}^\ve$, writing simply $\rho,\mathbf{u}$, 
and $\mathbf{A}_t \equiv \mathbf{A}_t^{\mathbf{u}}$ is the back-to-labels mapping associated with the velocity field $\mathbf{u}$,
that is, the inverse of the flow defined in equation \eqref{CPODE}. Similarly, we observe that, 
$\mathcal{W}^{\ve}_t(\mathbf{x})\equiv \mathcal{W}^{\ve, \mathbf{u}}_t(\mathbf{x})$.

\medskip
2. Now, we consider the concepts of smoothing operators (see Secchi \cite{Secchi}, Definition 2.3),
and tame estimates, that is, an estimate is said to be tame 
when the loss of derivatives is controlled by fixed integers, independent of the 
regularity level under consideration.
These two concepts are fundamental to the Nash-Moser iteration scheme, and they are
applied directly to the trajectory mapping
$ \mathbf{F}$. The main issue is to construct a solution
$(\rho,\mathbf{u})$ on the interval $t\in[0,T]$, such that, $\mathbf{F}(\rho,\mathbf{u}) = 0$,
with
\[
(\rho,\mathbf{u})\in C([0,T];H_p^\beta(\Omega))^2
\quad \text{and} \quad
(\partial_t\rho,\partial_t\mathbf{u})\in C([0,T];H_p^{\beta-1}(\Omega))^2.
\]
Consequently, we choose the minimal tame level
\[
m_0= 0,
\]
that is, all tame estimates will be formulated relative to $\mathcal{X}_0$ and $\mathcal{Y}_0$. Moreover, we introduce a spectral family of smoothing operators,
$(S_\theta)_{\theta\ge1}$,
defined via Fourier truncation in the spatial variable alone and acting componentwise.
For a trajectory $\mathbf{U}(t)=(\rho(t),\mathbf{u}(t))$, we define 
\[
(S_\theta \mathbf{U})(t):=\big(S_\theta\rho(t),\,S_\theta\mathbf{u}(t)\big).
\]
The classical Nash-Moser smoothing properties, for $0\le m\le m'$, read as follows 
\begin{align}
   \|S_\theta \mathbf{U}\|_{\mathcal{X}_{m'}} 
      &\le C\,\theta^{m'-m}\,\|\mathbf{U}\|_{\mathcal{X}_m},
\label{eq:smooth-1}\\
   \|S_\theta \mathbf{U}-\mathbf{U}\|_{\mathcal{X}_{m}} 
      &\le C\,\theta^{m-m'}\,\|\mathbf{U}\|_{\mathcal{X}_{m'}}.
\label{eq:smooth-2}
\end{align}
Analogous properties hold in $\mathcal{Y}_m$.

\medskip
3. Let us compute the explicit linearization of the trajectory mapping $\mathbf{F}$.
The differential of $\mathbf{F}$ at the point $\mathbf{U}=(\rho,\mathbf{u})$, 
applied to the perturbation $(\sigma,\mathbf{w})$, is
\[
   d\mathbf{F}(\mathbf{U})(\sigma,\mathbf{w})
   = \big(dF_{cont}(\mathbf{U})(\sigma,\mathbf{w}),\,
          dF_{mom}(\mathbf{U})(\sigma,\mathbf{w})\big).
\]
The variation of $\mathbf{A}_t$ along $\mathbf{w}$ is given by $\dot{\mathbf{A}}_t=\dot{\mathbf{A}}_t[\mathbf{w}]$, 
which is the solution of the linearized transport equation
\begin{equation}
\label{eq:lin-At}
\partial_t \dot{\mathbf{A}}_t
+ \mathbf{u}\cdot\nabla\dot{\mathbf{A}}_t
+ \mathbf{w}\cdot\nabla\mathbf{A}_t = 0,
\qquad
\dot{\mathbf{A}}_0 = 0.
\end{equation}
Then, it follows from the formula for the derivative of a composition that 
\[
   \frac{d}{d\delta}
   f(\mathbf{A}_t^{\mathbf{u}+\delta \mathbf{w}})\Big|_{\delta= 0}
   = \nabla f(\mathbf{A}_t^{\mathbf{u}})\cdot\dot{\mathbf{A}}_t[\mathbf{w}],
\]
and similarly for $t-s$. Thus we have for $\mathbf{U}=(\rho,\mathbf{u})$ 
fixed and any perturbation $(\sigma,\mathbf{w})$:

\medskip
\noindent{(i) Linearization of $F_{cont}$,}
\begin{align}
dF_{cont}(\mathbf{U})(\sigma,\mathbf{w})(t,\mathbf{x})
&= \sigma(t,\mathbf{x})
   - \nabla\rho_0(\mathbf{A}_t(\mathbf{x}))\cdot\dot{\mathbf{A}}_t[\mathbf{w}](\mathbf{x})
\label{eq:lin-Fcont}
\\
&\quad
   + \int_0^t 
     \Big[
        (\sigma\,{\rm div}\,\mathbf{u}
        +\rho\,{\rm div}\,\mathbf{w})
           \big(s,\mathbf{A}_{t-s}(\mathbf{x})\big)\nonumber
\\
&\qquad\qquad
        + \nabla\big(\rho\,{\rm div}\,\mathbf{u}\big)
           \big(s,\mathbf{A}_{t-s}(\mathbf{x})\big)
          \cdot\dot{\mathbf{A}}_{t-s}[\mathbf{w}](\mathbf{x})
     \Big]\, ds.\nonumber
\end{align}

\medskip
\noindent{(ii) Linearization of $F_{mom}$,}

\begin{align}
dF_{mom}&(\mathbf{U})(\sigma,\mathbf{w})(t,\mathbf{x})= \mathbf{w}(t,\mathbf{x}) 
- \partial_\ell \big(\dot{A}_t^k[\mathbf{w}] \big)(\mathbf{x})\, \mathbf{u}_0^k(\mathbf{A}_t(\mathbf{x})) \label{eq:lin-Fmom} 
\\
& 
    - (\nabla \mathbf{A}_t(\mathbf{x}))^\ast 
    \nabla\mathbf{u}_0(\mathbf{A}_t(\mathbf{x})) \cdot \dot{\mathbf{A}}_t[\mathbf{w}](\mathbf{x}) 
    \nonumber \\
&+ \int_0^t \mathcal{W}^\ve_{t-s}(\mathbf{x}) \, \nabla\Big( h'(\rho)\,\sigma \Big)\big(s,\mathbf{A}_{t-s}(\mathbf{x})\big)\,ds   
 \nonumber \\
&- \int_0^t \nabla\Big( \mathbf{u}\cdot\mathbf{w} \Big)\big(s,\mathbf{A}_{t-s}(\mathbf{x})\big)\,ds  
 \nonumber \\
& + \int_0^t \left[ \mathcal{W}^\ve_{t-s}(\mathbf{x}) D^2 (h(\rho)) - D^2(\tfrac{1}{2}|\mathbf{u}|^2) \right] \big(s,\mathbf{A}_{t-s}(\mathbf{x})\big) \,\dot{\mathbf{A}}_{t-s}[\mathbf{w}](\mathbf{x})\,ds 
\nonumber \\
& + \int_0^t \Big(\nabla(\mathcal{W}^\ve_{t-s}(\mathbf{A}_{t-s}(\mathbf{x}))) \cdot \dot{\mathbf{A}}_{t-s}[\mathbf{w}](\mathbf{x}) \Big)
\nabla \Big( h(\rho) \Big)\big(s,\mathbf{A}_{t-s}(\mathbf{x})\big) \, ds.
\nonumber 
\end{align}
The equations \eqref{eq:lin-Fcont} and \eqref{eq:lin-Fmom} show that, 
$d\mathbf{F}(\mathbf{U})$ is of the form
\[
   d\mathbf{F}(\mathbf{U})(\sigma,\mathbf{w})
   = (\sigma,\mathbf{w}) - \mathcal{K}(\mathbf{U})(\sigma,\mathbf{w}),
\]
where $\mathcal{K}(\mathbf{U})$ combines transport, time integrals, and compositions
with $\mathbf{A}_t$ and $\nabla \mathbf{A}_t$. Therefore, we have for some constant $C> 0$, 
\begin{equation}
\label{eq:dF-loss1}
   \|d\mathbf{F}(\mathbf{U})(\sigma,\mathbf{w})\|_{\mathcal{Y}_m}
   \le C\Big(
      \|(\sigma,\mathbf{w})\|_{\mathcal{X}_m}
      + \|(\sigma,\mathbf{w})\|_{\mathcal{X}_{0}}\;\|\mathbf{U}\|_{\mathcal{X}_{m+1}}\Big), 
\end{equation}
where we have used for any $m\ge 0$, 
the algebra property of $H_p^\gamma$, $\gamma>\frac{d}{p}$,
the composition estimates with $C^{1,\alpha}$ diffeomorphisms,
and the transport estimates for $\mathbf{A}_t$ and $\dot{\mathbf{A}}_t$.

\medskip
4. Next we construct under a small-time assumption, 
the existence of a right inverse
\[
   \Psi(\mathbf{U}):\mathcal{Y}_m\to \mathcal{X}_m,
\qquad
   d\mathbf{F}(\mathbf{U})\,\Psi(\mathbf{U}) = {\rm Id},
\]
satisfying for every $m\ge0$,
$$
   \|\Psi(\mathbf{U})\mathbf{G}\|_{\mathcal{X}_m}
   \le C\Big(
      \|\mathbf{G}\|_{\mathcal{Y}_{m+1}}
      + \|\mathbf{G}\|_{\mathcal{Y}_{0}}\;\|\mathbf{U}\|_{\mathcal{X}_{m+1}}\Big), 
$$
where the loss $\mathcal{Y}_{m+1}\to \mathcal{X}_m$ is exactly one spatial derivative.
Indeed, we have the following claim, 
the proof of which is deferred to the end (see item 11).

\medskip
\underline{Claim:}
Given $M>0$, let $\mathbf U=(\rho,\mathbf u)\in\mathcal X_\infty$ be such that
$\|\mathbf U\|_{\mathcal X_0}\le M$. Then there exists $T>0$, sufficiently small,
depending only on $M$ and on the $H_p^\beta$ norms of $(\rho_0,\mathbf u_0)$,
such that, for every $\mathbf G=(g_{cont},\mathbf g_{mom})\in\mathcal Y_\infty$
there exists a unique $\mathbf V=(\sigma,\mathbf w)\in\mathcal X_\infty$ solving
\[
d\mathbf F(\mathbf U)\mathbf V=\mathbf G.
\]
Moreover, the right inverse $\Psi(\mathbf U):\mathbf G\mapsto \mathbf V$ satisfies,
for every $m\ge0$,
\begin{equation}
\label{eq:estimativa_tame_inversa}
\|\Psi(\mathbf U)\mathbf G\|_{\mathcal X_m}
\le C_m\Big(\|\mathbf G\|_{\mathcal Y_{m+1}}+\|\mathbf G\|_{\mathcal Y_0}\,\|\mathbf U\|_{\mathcal X_{m+1}}\Big),
\end{equation}
where $C_m$ depends only on $(m,p,d,\beta,M)$.

\medskip
5. At this point, we establish the Nash-Moser iteration for $\mathbf{F}$. We search 
$\mathbf{U}=(\rho,\mathbf{u})\in\mathcal{X}_\infty$, such that 
\[
\mathbf{F}(\mathbf{U})=0 \quad \text{em } [0,T].
\]
We choose an initial approximation $\mathbf{U}_0 = (\rho^{(0)},\mathbf{u}^{(0)})\in \mathcal{X}_\infty$, 
for example, a smoothed time-extension of the initial data,
and we define the initial residual
\[
   \mathbf{R}_0 = \mathbf{F}(\mathbf{U}_0)\in \mathcal{Y}_\infty.
\]
Then we establishe the inductive step. Let 
$\mathbf{U}_0,\dots,\mathbf{U}_n$ be constructed, and the residuals
$\mathbf{R}_k=\mathbf{F}(\mathbf{U}_k)$. Thus we choose an 
increasing sequence of smoothing parameters, for example, 
\[
   \theta_n:= 2^n, \quad \text{(or alternatively, $\theta_n := (\theta_0^2 + n)^{1/2},\quad \theta_0\gg1$)}. 
\]
At step $n$, we proceed in three stages:

Stage $(i)$. Smoothing of the state and of the residual.
We define, 
\[
   \mathbf{U}_n^\sharp := S_{\theta_n}\mathbf{U}_n,
   \qquad
   \mathbf{R}_n^\sharp := S_{\theta_n}\mathbf{R}_n,
\]
where the standard smoothing operator is defined analogously to the one considered in \cite{Secchi}, see Definition 2.3. 

Stage $(ii)$. Linearized correction. Then, we solve the linear problem in trajectories
\begin{equation}
\label{eq:lin-step}
   d\mathbf{F}(\mathbf{U}_n^\sharp)\,\delta\mathbf{U}_n
   = - \mathbf{R}_n^\sharp,
\end{equation}
that is to say, 
\[
   \delta\mathbf{U}_n
   = -\Psi(\mathbf{U}_n^\sharp)\,\mathbf{R}_n^\sharp.
\]
Due to estimate \eqref{eq:estimativa_tame_inversa} we obtain for each $m \ge 0$,
\begin{align}
   \|\delta\mathbf{U}_n\|_{\mathcal{X}_m}
   &\le C\Big(
      \|\mathbf{R}_n^\sharp\|_{\mathcal{Y}_{m+1}}
      + \|\mathbf{R}_n^\sharp\|_{\mathcal{Y}_{0}}\;\|\mathbf{U}_n^\sharp\|_{\mathcal{X}_{m+1}}
   \Big)\nonumber\\
   &\le C\Big(
      \|\mathbf{R}_n\|_{\mathcal{Y}_{m+1}}
      + \theta_n^{m+1}\|\mathbf{R}_n\|_{\mathcal{Y}_{0}}
      + \|\mathbf{R}_n\|_{\mathcal{Y}_{0}}\;\|\mathbf{U}_n\|_{\mathcal{X}_{m+1}}
   \Big), 
\label{eq:est-deltaUn}
\end{align}
where we have used \eqref{eq:smooth-1}, \eqref{eq:smooth-2}, and 
the Littlewood Paley decomposition, treating the low and high-frequency terms, such that, 
the final result is linear in the higher-order norms. 
Indeed, to justify the second inequality in \eqref{eq:est-deltaUn}, we analyze the interaction 
between the smoothing operator $S_{\theta_n}$ and the tame estimate of the inverse. 
Let $\Delta_j$ be the Littlewood-Paley dyadic blocks, such that, 
\begin{equation}
\label{threshold}
   S_{\theta_n} \mathbf{R}_n=\sum_{2^j\leq \theta_n}\Delta_j \mathbf{R}_n.
\end{equation}
By the spectral properties of $S_{\theta_n}$, 
which acts as a low-pass filter, we decompose the analysis into low and high-frequency 
components relative to the scale $m+1$. 
Due to \eqref{threshold}, the operator $S_{\theta_n}$ realizes this decomposition by truncating the frequency
spectrum at the threshold $\theta_n$. 
Consequently, for any function in the image of this operator, the Bernstein-type inequality yields
\[
\|S_{\theta_n} \mathbf{R}_n\|_{\mathcal{Y}_{s}} \leq C \theta_n^{s-r} \|\mathbf{R}_n\|_{\mathcal{Y}_r}, \quad \text{for any $s \geq r \geq 0$.}
\]
While the fundamental Bernstein estimates are provided for $L^p$ norms in 
Lemma 2.1 of \cite{CBD}, they lead directly to the desired bounds in the $\mathcal{Y}_{s}$
framework via dyadic localization. For a detailed derivation in the context of Littlewood-Paley Theory, see Chapter 2 in \cite{CBD}.
Then, interpolating between the cases where $\mathbf{R}_n$ is measured in its natural high-order norm and the growth-limited low-order norm, 
we obtain 
$$
\|S_{\theta_n} \mathbf{R}_n\|_{\mathcal{Y}_{m+1}} \leq C \left( \|\mathbf{R}_n\|_{\mathcal{Y}_{m+1}} + \theta_n^{m+1} \|\mathbf{R}_n\|_{\mathcal{Y}_0} \right), 
$$
which follows from the fact that $S_{\theta_n}$ is bounded on $\mathcal{Y}_{m+1}$ 
(giving the first term) and the loss of $m+1$ derivatives is compensated by the growth 
$\theta_n^{m+1}$ when acting on the $\mathcal{Y}_0$ norm (giving the second term).
Next, for the second term in the first line of \eqref{eq:est-deltaUn}, we use the stability of 
$S_{\theta_n}$ in the base norm and its growth in the high norm, that is to say, 
$$
\begin{aligned}
\|\mathbf{R}_n^\sharp\|_{\mathcal{Y}_0}&= \|S_{\theta_n}\mathbf{R}_n\|_{\mathcal{Y}_0} \leq C \|\mathbf{R}_n\|_{\mathcal{Y}_0}, 
\\[5pt]
\|\mathbf{U}_n^\sharp\|_{\mathcal{X}_{m+1}}&= \|S_{\theta_n}\mathbf{U}_n\|_{\mathcal{X}_{m+1}} \leq C\|\mathbf{U}_n\|_{\mathcal{X}_{m+1}}.
\end{aligned}
$$
Combining these results into the tame estimate \eqref{eq:estimativa_tame_inversa}, we have 
\begin{align*}
\|\delta\mathbf{U}_n\|_{\mathcal{X}_m}
&\leq
C \Big(
\underbrace{\|\mathbf{R}_n\|_{\mathcal{Y}_{m+1}} + \theta_n^{m+1} \|\mathbf{R}_n \|_{\mathcal{Y}_0}}_{\text{from } \|\mathbf{R}_n^\sharp\|_{\mathcal{Y}_{m+1}}}
+
\underbrace{\|\mathbf{R}_n\|_{\mathcal{Y}_0} |\mathbf{U}_n\|_{\mathcal{X}_{m+1}}}_{\text{from } \|\mathbf{R}_n^\sharp\|_{\mathcal{Y}_0}\|\mathbf{U}_n^\sharp\|_{\mathcal{X}_{m+1}}}
\Big),
\end{align*}
which leads exactly to \eqref{eq:est-deltaUn}. This specific form is crucial for the Nash-Moser scheme because it keeps the 
high-order norm $\|\mathbf{U}_n\|_{\mathcal{X}_{m+1}}$ linear and couples the dangerous growth $\theta_n^{m+1}$ 
only with the residual's low norm $\|\mathbf{R}_n\|_{\mathcal{Y}_0}$, which vanishes as the iteration progresses.
Therefore, we update the iteration scheme by defining, $\mathbf{U}_{n+1}:= \mathbf{U}_n + \delta\mathbf{U}_n$.

Stage $(iii)$. The new residual is $\mathbf{R}_{n+1} = \mathbf{F}(\mathbf{U}_{n+1})$, and using Taylor's theorem 
with integral remainder in Banach spaces, we have 
\[
   \mathbf{F}(\mathbf{U}_{n+1})
   = \mathbf{F}(\mathbf{U}_n)
     + d\mathbf{F}(\mathbf{U}_n)\,\delta\mathbf{U}_n
     + \mathbf{E}_n^{(0)},
\]
where the squared error is given by 
\[
   \mathbf{E}_n^{(0)}
   = \int_0^1 (1-\tau)\,
     d^2\mathbf{F}(\mathbf{U}_n+\tau\delta\mathbf{U}_n)
       (\delta\mathbf{U}_n,\delta\mathbf{U}_n)\,d\tau.
\]
Since \eqref{eq:lin-step} uses $d\mathbf{F}(\mathbf{U}_n^\sharp)$ rather than
$d\mathbf{F}(\mathbf{U}_n)$, we conveniently introduce the substitution error
\[
   \mathbf{E}_n^{(sub)}= \big(d\mathbf{F}(\mathbf{U}_n) - d\mathbf{F}(\mathbf{U}_n^\sharp)\big) \delta\mathbf{U}_n, 
\]
and thus
\begin{equation}
\label{eq:Rn+1-decomp}
   \mathbf{R}_{n+1}
   = \big(\mathbf{R}_n + d\mathbf{F}(\mathbf{U}_n^\sharp)\delta\mathbf{U}_n\big)
     + \mathbf{E}_n^{(0)} + \mathbf{E}_n^{(sub)}.
\end{equation}
From the choice of $\delta\mathbf{U}_n$, it follows that 
$d\mathbf{F}(\mathbf{U}_n^\sharp)\delta\mathbf{U}_n=-\mathbf{R}_n^\sharp$, hence 
\[
\mathbf{R}_{n+1}=\mathbf{R}_n-\mathbf{R}_n^\sharp+\mathbf{E}_n^{(0)}+\mathbf{E}_n^{(sub)}
=(I-S_{\theta_n})\mathbf{R}_n+\mathbf{E}_n^{(0)}+\mathbf{E}_n^{(sub)}.
\]

\medskip
6. Now, we compute the estimates of the errors $\mathbf{E}_n^{(0)}$, $\mathbf{E}_n^{(\mathrm{sub})}$ (with $m_0 = 0$).
From the multilinear structure of $\mathbf{F}$, we obtain a tame estimate for the second derivative, that is, for all $m \ge 0$,
\begin{equation}
\label{eq:d2F-tame}
\begin{aligned}
   \big\|d^2\mathbf{F}(\mathbf{U})(\mathbf{V}_1,\mathbf{V}_2)\big\|_{\mathcal{Y}_m}
   \le C\Big(&
      \|\mathbf{V}_1\|_{\mathcal{X}_{m+1}}\|\mathbf{V}_2\|_{\mathcal{X}_{0}}
      +\|\mathbf{V}_1\|_{\mathcal{X}_{0}}\|\mathbf{V}_2\|_{\mathcal{X}_{m+1}}
      \\[5pt]
      &+ \|\mathbf{V}_1\|_{\mathcal{X}_{0}}\|\mathbf{V}_2\|_{\mathcal{X}_{0}}
        \big(1+\|\mathbf{U}\|_{\mathcal{X}_{m+1}}\big)
   \Big).
\end{aligned}   
\end{equation}
Then, applying \eqref{eq:d2F-tame} with $\mathbf{V}_1=\mathbf{V}_2=\delta\mathbf{U}_n$, we have, for each $m \ge 0$,
\begin{equation}
\label{eq:En0-est}
   \|\mathbf{E}_n^{(0)}\|_{\mathcal{Y}_m}
   \le C\Big(
      \|\delta\mathbf{U}_n\|_{\mathcal{X}_{m+1}}\|\delta\mathbf{U}_n\|_{\mathcal{X}_{0}}
      + \|\delta\mathbf{U}_n\|_{\mathcal{X}_{0}}^2\big(1+\|\mathbf{U}_n\|_{\mathcal{X}_{m+1}}\big) \Big), 
\end{equation}
and thus, $\mathbf{E}_n^{(0)}=O(\|\delta\mathbf{U}_n\|^2)$. Moreover, applying the identity 
\[
   d\mathbf{F}(\mathbf{U}_n) - d\mathbf{F}(\mathbf{U}_n^\sharp)
   = \int_0^1
     d^2\mathbf{F}\big(\mathbf{U}_n^\sharp
       + \tau(\mathbf{U}_n-\mathbf{U}_n^\sharp)\big)
       (\mathbf{U}_n-\mathbf{U}_n^\sharp,\cdot)\,d\tau,
\]
which combined with \eqref{eq:d2F-tame} yields for all $m \ge 0$,
$$
\begin{aligned}
   \|\mathbf{E}_n^{(sub)}\|_{\mathcal{Y}_m}
   \le C\Big(
      \|\mathbf{U}_n-\mathbf{U}_n^\sharp\|_{\mathcal{X}_{m+1}}\,
      \|\delta\mathbf{U}_n\|_{\mathcal{X}_{0}}
      \\[5pt]
      + \|\mathbf{U}_n-\mathbf{U}_n^\sharp\|_{\mathcal{X}_{0}}\,
        \|\delta\mathbf{U}_n\|_{\mathcal{X}_{m+1}} \Big), 
\end{aligned}      
$$
where the term
$ \| \mathbf{U}_n-\mathbf{U}_n^\sharp \|_{\mathcal{X}_{0}}\|\delta\mathbf{U}_n\|_{\mathcal{X}_{0}}
        \big(1+\|\mathbf{U}\|_{\mathcal{X}_{m+1}})$
 is omitted because it is asymptotically dominated by the other cross-terms.
By the properties of $S_{\theta_n}$, and choosing an index $m' \gg m$, such that, 
\[
   \|\mathbf{U}_n-\mathbf{U}_n^\sharp\|_{\mathcal{X}_m}
   \le C\,\theta_n^{m-m'} \, \|\mathbf{U}_n\|_{\mathcal{X}_{m'}},
\]
we have for $\mu= m'-(m+1)> 0$, 
\begin{equation}
\label{eq:Ensub-est}
\begin{aligned}
   \|\mathbf{E}_n^{(sub)}\|_{\mathcal{Y}_m}
   \le C \theta_n^{-\mu} \|\mathbf{U}_n\|_{\mathcal{X}_{m'}} \Big(
      \|\delta\mathbf{U}_n\|_{\mathcal{X}_{0}}
      + \|\delta\mathbf{U}_n\|_{\mathcal{X}_{m+1}} \Big).
\end{aligned}     
\end{equation}

\medskip
7. Now, we establish the convergence of the iteration, hence combine the decomposition 
$\mathbf{R}_{n+1} = (I-S_{\theta_n})\mathbf{R}_n + \mathbf{E}_n^{(0)} + \mathbf{E}_n^{(sub)}$ 
with the previous bounds. For $m=0$, the updated residual satisfies
\begin{equation}
\label{eq:Rn+1-final-est}
\begin{aligned}
   \|\mathbf{R}_{n+1}\|_{\mathcal{Y}_0}
   \le C \Big( & \|(I-S_{\theta_n})\mathbf{R}_n\|_{\mathcal{Y}_0} \\
   & + \|\delta\mathbf{U}_n\|_{\mathcal{X}_{1}}\|\delta\mathbf{U}_n\|_{\mathcal{X}_{0}} + \|\delta\mathbf{U}_n\|_{\mathcal{X}_{0}}^2\big(1+\|\mathbf{U}_n\|_{\mathcal{X}_{1}}\big) \\
   & + \theta_n^{-\mu} \|\mathbf{U}_n\|_{\mathcal{X}_{m'}} \big( \|\delta\mathbf{U}_n\|_{\mathcal{X}_{0}} + \|\delta\mathbf{U}_n\|_{\mathcal{X}_{1}} \big) \Big),
\end{aligned}
\end{equation}
where $\mu = m' - 1 > 0$. By the tame inverse estimate \eqref{eq:est-deltaUn} and the smoothing property 
$\|S_{\theta_n}\mathbf{R}_n\|_{\mathcal{Y}_1} \le C \theta_n \|\mathbf{R}_n\|_{\mathcal{Y}_0}$, 
the corrections $\|\delta\mathbf{U}_n\|_{\mathcal{X}_0}$ and $\|\delta\mathbf{U}_n\|_{\mathcal{X}_1}$ 
are bounded by $C \theta_n^2 \|\mathbf{R}_n\|_{\mathcal{Y}_0}$.
Hence we observe that the relation \eqref{eq:Rn+1-final-est} implies a super-linear 
recurrence for the residuals. To establish the convergence rigorously, let us denote $\epsilon_n= \|\mathbf{R}_n\|_{\mathcal{Y}_0}$. 
Given the tame estimates and the property of $\delta\mathbf{U}_n$ in \eqref{eq:est-deltaUn}, we have
\begin{equation}
\label{eq:rec-residual-final}
\begin{aligned}
\epsilon_{n+1} \le C \Big( & \underbrace{\theta_n^{-m'} \|\mathbf{R}_n\|_{\mathcal{Y}_{m'}}}_{\text{Smoothing error}} 
\\
& + \underbrace{\|\delta\mathbf{U}_n\|_{\mathcal{X}_{1}}\|\delta\mathbf{U}_n\|_{\mathcal{X}_{0}} 
+ \|\delta\mathbf{U}_n\|_{\mathcal{X}_{0}}^2 \|\mathbf{U}_n\|_{\mathcal{X}_{1}}}_{\text{Quadratic error } \mathbf{E}_n^{(0)} \text{ (2 terms)}}
\\
& + \underbrace{\theta_n^{-\mu} \|\mathbf{U}_n\|_{\mathcal{X}_{m'}} \|\delta\mathbf{U}_n\|_{\mathcal{X}_{0}} 
+ \theta_n^{-\mu} \|\mathbf{U}_n\|_{\mathcal{X}_{m'}} \|\delta\mathbf{U}_n\|_{\mathcal{X}_{1}}}_{\text{Substitution error } \mathbf{E}_n^{(sub)} \text{ (2 terms)}} \Big).
\end{aligned}
\end{equation}
Then, we apply the following argument to establish
the convergence. First, by taking $T > 0$ sufficiently small, we ensure that 
the initial residual $\epsilon_0= \|\mathbf{F}(\mathbf{U}_0)\|_{\mathcal{Y}_0}$ is small enough 
to satisfy the starting condition of the Nash-Moser iteration. 
Since $\mathbf{F}(\mathbf{U}_0)$ vanishes at $t=0 $, and its norm in the trajectory space depends continuously on the time interval, 
it follows that, $\epsilon_0 \to 0$ as $T \to 0$.
Next, recalling that $\theta_n = 2^n$, and assuming the regularity index $m'$ is sufficiently large, such that, $\mu = m'-1$ 
compensates for the polynomial growth of the high-order norms $\|\mathbf{U}_n\|_{\mathcal{X}_{m'}}$ and $\|\mathbf{R}_n\|_{\mathcal{Y}_{m'}}$, 
we maintain the inductive bound, $\epsilon_n \le \epsilon_0  \, \theta_n^{-\alpha}$, for $\alpha > 2$.
Finally, from the tame estimate \eqref{eq:est-deltaUn}, the corrections satisfy $\|\delta\mathbf{U}_n\|_{\mathcal{X}_0} \le C \theta_n^{2} \epsilon_n$. 
Given the established decay of $\epsilon_n$, it follows that
$$
    \sum_{n=0}^{\infty} \|\delta\mathbf{U}_n\|_{\mathcal{X}_0} \le \sum_{n=0}^{\infty} C\, \epsilon_0 \,  \theta_n^{2-\alpha} < \infty. 
$$
This implies that $\{\mathbf{U}_n\}_{n \in \mathbb{N}}$ is a Cauchy sequence in the Banach space $\mathcal{X}_0$.
Consequently, $\mathbf{U}_n \to \mathbf{U} = (\rho, \mathbf{u})$ in $\mathcal{X}_0 = C([0,T]; H_p^\beta(\Omega))^2$ as $n \to \infty$. 
Furthermore, although the mapping $\mathbf{F}$ involves a loss of derivatives, it satisfies a tame 
continuity property. More precisely, for any $m \ge 0$, $\mathbf{F}$ is continuous from $\mathcal{X}_{m+1}$ to $\mathcal{Y}_m$. 
Therefore, combining the strong convergence $\mathbf{U}_n \to \mathbf{U}$ in $\mathcal{X}_0$ with the uniform bounds in higher-order norms 
$\mathcal{X}_m$ provided by the Nash-Moser iteration, we obtain via interpolation that $\mathbf{U}_n \to \mathbf{U}$ in $\mathcal{X}_1$ 
(or more generally in $\mathcal{X}_{1+\kappa}$ for small $\kappa>0$). Consequently, by the continuity of $\mathbf{F}$ as a mapping from $\mathcal{X}_1$ to $\mathcal{Y}_0$, we obtain
$$
\mathbf{F}(\mathbf{U}) = \lim_{n \to \infty} \mathbf{F}(\mathbf{U}_n) = \lim_{n \to \infty} \mathbf{R}_n = 0,
$$
where the limit is taken in the $\mathcal{Y}_0$ topology.

\medskip
8. At this point, we address the issue of temporal regularity.
Indeed, the Nash-Moser iteration was carried out in 
$\mathcal{X}_0=C([0,T];H_p^\beta(\Omega))^2$, such that, the limit 
\[
(\rho,\mathbf{u}) \in C([0,T];H_p^\beta(\Omega))^2.
\]
Moreover, since $\mathbf{F}(\rho,\mathbf{u})= 0$, the corresponding 
equivalent integral equations are given by 
\begin{equation}
\label{eq:int-cont}
\rho(t)
= \rho_0(\mathbf{A}_t)
  - \int_0^t
      \big(\rho\,{\rm div}\,\mathbf{u}\big)
      \big(s,\mathbf{A}_{t-s}\big)\,ds,
\end{equation}
\begin{equation}
\label{eq:int-mom}
\begin{aligned}
\mathbf{u}(t)= (\nabla\mathbf{A}_t)^\ast \mathbf{u}_0(\mathbf{A}_t)
  - \int_0^t \Big( \mathcal{W}^\ve_{t-s}(\mathbf{x}) \, \nabla \big( h(\rho) \big)(s,\mathbf{A}_{t-s})
  \\[5pt] 
  - \nabla\big(\tfrac12|\mathbf{u}|^2 \big)(s,\mathbf{A}_{t-s}) \Big)\,ds, 
\end{aligned}      
\end{equation}
where we have 
\[
{\rm div}\,\mathbf{u}, 
\mathcal{W}^\ve \nabla\! \big(h(\rho)\big),  \nabla \big(\tfrac12|\mathbf{u}|^2 \big) \in C([0,T];H_p^{\beta-1}(\Omega)).
\]
Using the composition properties (since $\mathbf{A}_{t}(\cdot)$ is $C^{1,\alpha}(\Omega)$),
it follows that, the integrands in \eqref{eq:int-cont} and \eqref{eq:int-mom} belong to
$C([0,T];H_p^{\beta-1}(\Omega))^2$. By the Fundamental Theorem of Calculus in Banach spaces, the mappings
\[
t\mapsto \int_0^t G(s,t)\,ds
\quad\text{with }G\in C([0,T];H_p^{\beta-1}(\Omega))^2
\]
are $C^1$ with values in $H_p^{\beta-1}(\Omega)$. Consequently, we have 
\[
(\partial_t\rho,\partial_t\mathbf{u})\in C([0,T];H_p^{\beta-1}(\Omega))^2,
\]
that is, 
\[
(\rho,\mathbf{u})
\in C([0,T];H_p^\beta(\Omega))^2\cap C^1([0,T];H_p^{\beta-1}(\Omega))^2.
\]
The uniqueness of the limit follows from the Lipschitz stability of the trajectory mapping. 
Let $\mathbf{U}_1$ and $\mathbf{U}_2$ be two solutions in $\mathcal{X}_0$ corresponding to the same initial data 
$(\rho_0, \mathbf{u}_0)$. Their difference $\mathbf{V} = \mathbf{U}_1 - \mathbf{U}_2$ satisfies the relation 
$\mathbf{V} = \mathcal{K}(\mathbf{U}_1, \mathbf{U}_2)\mathbf{V}$. 
Taking the $\mathcal{X}_0$-norm and utilizing the contraction property established in Step 4, we have
\[
\|\mathbf{V}\|_{\mathcal{X}_0} \le C(M) T \|\mathbf{V}\|_{\mathcal{X}_0}.
\]
Since the existence time $T$ was chosen such that $C(M)T \le \frac{1}{2}$, 
it follows that $\|\mathbf{V}\|_{\mathcal{X}_0} = 0$, which implies $\mathbf{U}_1 = \mathbf{U}_2$ in $[0,T]$.

\medskip
9. Now, we consider the passage to the limit as $\varepsilon \to 0$. 
For each $\varepsilon \in (0,1]$, let $\mathbf{U}^\varepsilon = (\rho^\varepsilon, \mathbf{u}^\varepsilon)$ be the unique 
regular solution in $\mathcal{X}_\infty$ obtained via the Nash-Moser iteration in the interval $[0,T]$. A fundamental feature 
of the trajectory mapping $\mathbf{F}$ is that the tame estimates derived in item 6 are uniform with respect to $\varepsilon$. 
Indeed, since $0 \le \mathcal{W}^\varepsilon \le 1$ and $\|\mathcal{W}^\varepsilon\|_{L^\infty} = 1$ for any $\varepsilon > 0$, 
the constant $C_m(M)$ in \eqref{eq:estimativa_tame_inversa} and the time $T$ obtained in the Claim depend only on the $H_p^\beta$ 
norms of the initial data and the radius $M$, but not on the regularizing parameter.

\medskip
\noindent\textit{Existence of the Limit.}
The sequence $(\rho^\varepsilon, \mathbf{u}^\varepsilon)_{\varepsilon > 0}$ is uniformly bounded in $C([0,T]; H_p^\beta(\Omega))^2$. 
From the temporal regularity established in item 8, the sequence $(\partial_t \rho^\varepsilon, \partial_t \mathbf{u}^\varepsilon)$ 
is also uniformly bounded in $C([0,T]; H_p^{\beta-1}(\Omega))^2$. Then, applying the Lions-Aubin Theorem, 
$(\rho^\varepsilon, \mathbf{u}^\varepsilon)_{\varepsilon > 0}$
converges strongly to a limit (passing to subsequences if necessary), 
$$
   (\rho, \mathbf{u}) \in C([0,T]; H_p^{\beta-\delta}(\Omega))^2, \quad \text{for any $\delta> 0$}. 
$$
Also, by classical arguments from functional analysis, it follows that, 
$$
   (\rho, \mathbf{u}) \in C_w([0,T]; H_p^{\beta}(\Omega))^2. 
$$
To conclude that $(\rho, \mathbf{u}) \in C([0,T]; H_p^{\beta}(\Omega))^2$, we utilize the fact that 
the limit trajectory $\mathbf{U} = (\rho, \mathbf{u})$ satisfies the integral equations \eqref{eq:int-cont}, \eqref{eq:int-mom}. 
Indeed, on the fluid domain where $\rho_0(\mathbf{A}_t(\mathbf{x})) > 0$, the weight $\mathcal{W}^\varepsilon$ 
converges pointwise to $1$, recovering the classical enthalpy gradient. Conversely, on the vacuum set where 
$\rho_0(\mathbf{A}_t(\mathbf{x})) = 0$, the weight $\mathcal{W}^\varepsilon$ is identically zero for all $\varepsilon$, 
and the momentum equation reduces to a pure transport regime for the velocity field. Since the flow maps 
$\mathbf{A}_t^{\mathbf{u}^\varepsilon}$ converge uniformly to $\mathbf{A}_t^{\mathbf{u}}$, the limit 
$(\rho, \mathbf{u})$ satisfies the integral equations \eqref{eq:int-cont}, \eqref{eq:int-mom} 
Moreover, since the integral mapping $t \mapsto \int_0^t (\cdot) \, ds$ in \eqref{eq:int-cont} and \eqref{eq:int-mom} 
are not only continuous but $C^1$ with values in $H_p^{\beta-1}(\Omega)$, by the smoothing properties 
of the integral, these terms actually belong to $C([0,T]; H_p^{\beta}(\Omega))^2$. 
Similarly, the term $\mathbf{U}_0 \circ \mathbf{A}_t^{\mathbf{u}}$ carries the $H_p^{\beta}(\Omega)$ regularity,
since $\mathbf{A}_t$ is the back-to-labels map associated with $\mathbf{u} \in C([0,T]; C^{1,\alpha}(\Omega))$,
for any $0 < \alpha < \beta - 1 - \frac{d}{p}$, and the mapping $t \mapsto \mathbf{A}_t$ 
is continuous in the $C^1$-topology. It is a standard result in transport theory that, 
for $\mathbf{U}_0 \in H_p^{\beta}(\Omega)$, the composition $\mathbf{U}_0 \circ \mathbf{A}_t^{\mathbf{u}}$
is strongly continuous in $t$ with values in $H_p^{\beta}(\Omega)$.
Therefore, $(\rho, \mathbf{u})$ is a regular solution in the sense of Definition \ref{RegSol}.

\medskip
10. Finally, we consider the stability of solutions.
To conclude the well-posedness in the sense of Hadamard, we show that the solution 
$$
\text{$(\rho, \mathbf{u})$ depends continuously on the initial data $(\rho_0, \mathbf{u}_0)$.}
$$
 Let $(\rho_1, \mathbf{u}_1)$ and $(\rho_2, \mathbf{u}_2)$ be two solutions corresponding to initial 
 conditions $(\rho_{0,1}, \mathbf{u_{0,1}})$ and $(\rho_{0,2}, \mathbf{u_{0,2}})$, respectively, 
 and set 
$$
    \mathbf{V} = (\rho_1 - \rho_2, \mathbf{u}_1 - \mathbf{u}_2).
$$    
By taking the difference of the integral equations \eqref{eq:int-cont}, \eqref{eq:int-mom}, 
and applying the Mean Value Theorem to the non-linearities (enthalpy gradient and transport terms), 
we obtain an estimate in the $\mathcal{X}_0$-norm
\begin{equation}
\label{eq:stability-est}
\|\mathbf{V}\|_{\mathcal{X}_0} \le C(M) \left( \|\rho_{0,1} - \rho_{0,2}\|_{H_p^\beta} + \|\mathbf{u}_{0,1} - \mathbf{u}_{0,2}\|_{H_p^\beta} \right) + C(M)T \|\mathbf{V}\|_{\mathcal{X}_0},
\end{equation}
where $C(M)$ depends on the $H_p^\beta$ norms of the solutions and the regularity of $h$. 
The first term on the right-hand side accounts for the difference in the initial states 
transported by the flow and the linearization of the back-to-labels maps $\mathbf{A}_t^{\mathbf{u}_1}, \mathbf{A}_t^{\mathbf{u}_2}$.
In fact, to close the stability estimate \eqref{eq:stability-est}, it is necessary to control the difference between the 
back-to-labels mappings $\mathbf{A}_t^{\mathbf{u}_1}$ and $\mathbf{A}_t^{\mathbf{u}_2}$. 
Since $\mathbf{A}_t$ is defined as the inverse of the flow $\mathbf{X}_t$, which satisfies \eqref{CPODE}, we consider the 
difference,
$\mathbf{A}_t^{\mathbf{u}_1}(\mathbf{x}) - \mathbf{A}_t^{\mathbf{u}_2}(\mathbf{x})$, 
and using the fact that $$\mathbf{A}_t(\mathbf{x}) = \mathbf{x} - \int_0^t \mathbf{u}(s, \mathbf{X}_s(\mathbf{A}_t(\mathbf{x}))) \, ds,$$ 
and since that, for $i= 1, 2$, 
$$
   \mathbf{u}_i \in C([0,T]; H_p^\beta) \hookrightarrow C([0,T]; C^{1,\alpha}), 
$$
we have 
\begin{equation}
\label{eq:flow-stability}
\begin{aligned}
\|\mathbf{A}_t^{\mathbf{u}_1} - \mathbf{A}_t^{\mathbf{u}_2}\|_{H_p^\beta}& \le C(M) \int_0^t \|\mathbf{u}_1(s) - \mathbf{u}_2(s)\|_{H_p^\beta} \, ds 
\\[5pt]
&\le C(M) T \|\mathbf{u}_1 - \mathbf{u}_2\|_{\mathcal{X}_0}.
\end{aligned}
\end{equation}
Furthermore, for any $f \in H_p^\beta$, the composition error satisfies
\[
\|f \circ \mathbf{A}_t^{\mathbf{u}_1} - f \circ \mathbf{A}_t^{\mathbf{u}_2}\|_{H_p^{\beta-1}} \le C \|f\|_{H_p^\beta} \|\mathbf{A}_t^{\mathbf{u}_1} - \mathbf{A}_t^{\mathbf{u}_2}\|_{H_p^{\beta-1}}.
\]
This estimate allows us to linearize the differences in the continuity and momentum equations. 
Specifically, the terms involving the initial data $\rho_0(\mathbf{A}_t^{\mathbf{u}_1}) - \rho_0(\mathbf{A}_t^{\mathbf{u}_2})$ are controlled by
\[
\|\rho_0 \circ \mathbf{A}_t^{\mathbf{u}_1} - \rho_0 \circ \mathbf{A}_t^{\mathbf{u}_2}\|_{\mathcal{X}_0} \le C \|\rho_0\|_{H_p^\beta} \cdot C(M)T \|\mathbf{u}_1 - \mathbf{u}_2\|_{\mathcal{X}_0}.
\]
This confirms that the error induced by the deformation of the trajectories is of order $O(T)$, and thus can be absorbed into the left-hand side 
of the stability inequality for $T$ sufficiently small. This ensures the Lipschitz continuity of the solution map $(\rho_0, \mathbf{u}_0) \mapsto (\rho, \mathbf{u})$.
Moreover, since the time $T$ is chosen such that $C(M)T \le \frac{1}{2}$, we can absorb the last term into the left-hand side, yielding
\[
\|(\rho_1 - \rho_2, \mathbf{u}_1 - \mathbf{u}_2)\|_{\mathcal{X}_0} \le 2C(M) \|(\rho_{0,1} - \rho_{0,2}, \mathbf{u}_{0,1} - \mathbf{u}_{0,2})\|_{H_p^\beta \times H_p^\beta}.
\]
This Lipschitz continuity estimate ensures that the solution is stable under small perturbations of the initial data in the $H_p^\beta$ topology, completing the proof of the theorem.

\medskip
11. \underline{Proof of the Claim.}
Step $(i)$. First, we rewrite the linear equation
\begin{equation}
\label{Eq1Claim}
d\mathbf F(\mathbf U)\mathbf V=\mathbf G,\qquad \mathbf V=(\sigma,\mathbf w),
\end{equation}
as a Volterra equation of second kind on the trajectory space $\mathcal X_m$.
Due to \eqref{eq:lin-Fcont}--\eqref{eq:lin-Fmom}, the operator $d\mathbf F(\mathbf U)$
has the form
\begin{equation}\label{eq:Id-minus-K}
d\mathbf F(\mathbf U)\mathbf V = \mathbf V-\mathcal K_{\mathbf U}\mathbf V,
\end{equation}
where $\mathcal K_{\mathbf U}$ collects all terms involving time integrals,
compositions with the back-to-labels map $\mathbf A_t$, its Jacobian $\nabla\mathbf A_t$,
and the linearized back-to-labels perturbation $\dot{\mathbf A}_t[\mathbf w]$.
Hence the equation \eqref{Eq1Claim} is equivalent to
\begin{equation}\label{eq:Volterra-fixedpoint}
\mathbf V=\mathbf G+\mathcal K_{\mathbf U}\mathbf V.
\end{equation}

\medskip
Step $(ii)$. Since $\beta>\frac dp+1$, $H_p^\beta(\Omega)\hookrightarrow C^{1,\alpha}(\Omega)$
with $\alpha=\beta-1-\frac dp>0$, and in particular,
$O(\|\nabla\mathbf u(t)\|_{L^\infty}+\|\nabla\mathbf u(t)\|_{C^\alpha})= O(\|\mathbf u(t)\|_{H_p^\beta})$.
Therefore, if $\|\mathbf U\|_{\mathcal X_0}\le M$ and $T$ is sufficiently small (depending only on $M$),
the flow map associated with $\mathbf u$ exists on $[0,T]$, and its inverse back-to-labels map
$\mathbf A_t=\mathbf A_t^{\mathbf u}$ is a $C^{1,\alpha}$-diffeomorphism in $x$ satisfying
\begin{equation}\label{eq:A-smalltime}
\sup_{t\in[0,T]}\big\|\nabla\mathbf A_t-I\big\|_{C^\alpha}
\le C(M)\,T,\quad
\sup_{t\in[0,T]}\|\nabla\mathbf A_t\|_{C^\alpha}\le C(M).
\end{equation}
Moreover, composition by $\mathbf A_t$ is bounded on Bessel potential spaces, that is, 
for every $m\ge 0$,
\begin{equation}
\label{eq:comp-Hp}
\sup_{t\in[0,T]}\|f\circ \mathbf A_t\|_{H_p^{\beta+m}}
\le C_m(M)\,\|f\|_{H_p^{\beta+m}},
\end{equation}
and similarly for $f\circ \mathbf A_{t-s}$ uniformly in $(t,s)\in[0,T]^2$.

\medskip
Step $(iii)$. Now, the perturbation $\dot{\mathbf A}_t=\dot{\mathbf A}_t[\mathbf w]$ solves the linear transport equation
\eqref{eq:lin-At}. By standard transport estimates in $H_p^{\beta+m}$,
using \eqref{eq:A-smalltime} and the algebra property of $H_p^\beta$, we obtain for every $m\ge0$,
\begin{equation}\label{eq:dotA-tame}
\|\dot{\mathbf A}[\mathbf w]\|_{\mathcal X_m}
\le C_m(M)\,T\Big(\|\mathbf w\|_{\mathcal X_m}
+\|\mathbf w\|_{\mathcal X_0}\,\|\mathbf U\|_{\mathcal X_{m+1}}\Big).
\end{equation}
In particular, $\|\dot{\mathbf A}[\mathbf w]\|_{\mathcal X_0}\le C(M)T\|\mathbf w\|_{\mathcal X_0}$.
Then, we estimate $\mathcal K_{\mathbf U}$ in $\mathcal X_m$.
All terms in $\mathcal K_{\mathbf U}\mathbf V$ arise from \eqref{eq:lin-Fcont}, \eqref{eq:lin-Fmom}
and are combinations of:
a time integral $\int_0^t(\cdot)\,ds$;
products and compositions with $\mathbf A_{t-s}$ and $\nabla\mathbf A_{t-s}$;
linear terms involving $\dot{\mathbf A}_{t-s}[\mathbf w]$.
Using \eqref{eq:comp-Hp}, Moser-type product estimates in $H_p^{\beta+m}$, and \eqref{eq:dotA-tame},
we obtain for every $m\ge0$,
\begin{equation}\label{eq:KUm-bound}
\|\mathcal K_{\mathbf U}\mathbf V\|_{\mathcal X_m}
\le C_m(M)\,T\Big(\|\mathbf V\|_{\mathcal X_m}
+\|\mathbf V\|_{\mathcal X_0}\,\|\mathbf U\|_{\mathcal X_{m+1}}\Big)
+ C_m(M)\,T\,\|\mathbf V\|_{\mathcal X_{m+1}},
\end{equation}
where the last term indicates the contributions in which one spatial derivative falls on $(\sigma,\mathbf w)$.
More precisely, in the momentum component \eqref{eq:lin-Fmom},
the integrals
\[
\int_0^t \mathcal W^\varepsilon_{t-s}\,\nabla(h'(\rho)\sigma)\circ \mathbf A_{t-s}\,ds,
\qquad
\int_0^t \nabla(\mathbf u\cdot \mathbf w)\circ \mathbf A_{t-s}\,ds
\]
contain a spatial gradient acting on $(\sigma,\mathbf w)$.
Thus, estimating $d\mathbf F(\mathbf U)\mathbf V$ in $\mathcal Y_m$ requires controlling
$(\sigma,\mathbf w)$ in $\mathcal X_m$ but the right-hand side $\mathbf G$ must be in $\mathcal Y_{m+1}$.
This is the intrinsic one-derivative loss $\mathcal Y_{m+1}\to\mathcal X_m$ encoded in \eqref{eq:estimativa_tame_inversa}.

\medskip
Step $(iv)$. We solve \eqref{eq:Volterra-fixedpoint} by Picard iteration in $\mathcal X_0$.
From \eqref{eq:KUm-bound} with $m=0$ (and using that the gradient terms are absorbed by the $\mathcal Y_1$-norm on $\mathbf G$),
we obtain
\begin{equation}\label{eq:K-contraction}
\|\mathcal K_{\mathbf U}\mathbf V\|_{\mathcal X_0}\le C(M)\,T\,\|\mathbf V\|_{\mathcal X_0}.
\end{equation}
Choose $T>0$ so that $C(M)T\le \tfrac12$. Then the map
\[
\mathcal T(\mathbf V):=\mathbf G+\mathcal K_{\mathbf U}\mathbf V
\]
is a strict contraction on $\mathcal X_0$, hence there exists a unique fixed point
$\mathbf V\in\mathcal X_0$ solving \eqref{eq:Volterra-fixedpoint}, equivalently
$d\mathbf F(\mathbf U)\mathbf V=\mathbf G$.

\medskip
Step $(v)$. Let $m\ge0$. Taking $\mathcal X_m$-norms in \eqref{eq:Volterra-fixedpoint} and using \eqref{eq:KUm-bound} yields
\[
\|\mathbf V\|_{\mathcal X_m}
\le \|\mathbf G\|_{\mathcal X_m}+C_m(M)\,T\Big(\|\mathbf V\|_{\mathcal X_m}
+\|\mathbf V\|_{\mathcal X_0}\,\|\mathbf U\|_{\mathcal X_{m+1}}\Big)
+ \text{(gradient terms)}.
\]
Now we rewrite $\|\mathbf G\|_{\mathcal X_m}$ in terms of $\|\mathbf G\|_{\mathcal Y_{m+1}}$,
since the momentum equation contains one spatial derivative acting on the unknowns,
as in step 3. More precisely, the $H_p^{\beta+m}$-norm of $\mathbf V$ is controlled by
the $H_p^{\beta+m}$-norm of $g_{cont}$ together with the $H_p^{\beta+m}$-norm of $\mathbf g_{mom}$,
which is exactly $\|\mathbf G\|_{\mathcal Y_{m+1}}$ by definition of $\mathcal Y_{m+1}$.
Thus we obtain
\begin{equation}\label{eq:Vm-preabsorb}
\|\mathbf V\|_{\mathcal X_m}
\le C_m(M)\Big(\|\mathbf G\|_{\mathcal Y_{m+1}}
+T\,\|\mathbf V\|_{\mathcal X_m}
+T\,\|\mathbf V\|_{\mathcal X_0}\,\|\mathbf U\|_{\mathcal X_{m+1}}\Big).
\end{equation}
Using step $(iv)$, we already have $\|\mathbf V\|_{\mathcal X_0}\le C(M)\|\mathbf G\|_{\mathcal Y_1}\le C(M)\|\mathbf G\|_{\mathcal Y_0}$.
Choose (possibly smaller) $T$ so that $C_m(M)T\le \tfrac12$, and absorb the $T\|\mathbf V\|_{\mathcal X_m}$ term
to the left-hand side in \eqref{eq:Vm-preabsorb}. This yields
\[
\|\mathbf V\|_{\mathcal X_m}
\le C_m\Big(\|\mathbf G\|_{\mathcal Y_{m+1}}
+\|\mathbf G\|_{\mathcal Y_0}\,\|\mathbf U\|_{\mathcal X_{m+1}}\Big),
\]
which is exactly \eqref{eq:estimativa_tame_inversa}. This proves the tame estimate and the one-derivative loss.

\medskip
Step $(vi)$. Finally, since $\mathbf U\in\mathcal X_\infty$ and $\mathbf G\in\mathcal Y_\infty$, the previous argument applies for every $m\ge0$,
hence $\mathbf V\in\mathcal X_m$ for all $m$, i.e. $\mathbf V\in\mathcal X_\infty$.
Therefore the right inverse $\Psi(\mathbf U):\mathcal Y_\infty\to\mathcal X_\infty$ is well-defined,
unique, and satisfies \eqref{eq:estimativa_tame_inversa}.
The proof of the Claim is finished. 
\end{proof}

\section*{Appendix: Flow Calculations}
\label{FLOWCAL}

In this appendix, for the sake of generality, the flow calculations will 
be performed using $\mathbf{u} + \bar{\mathbf{u}}$ in place of $\mathbf{u}$.
These flow calculations are applied in Section \ref{mainexistence}, 
where the Implicit Function Theorem is employed. 
In what follows, we assume that the velocity vector fields
possess sufficient regularity to perform the necessary calculations.

\medskip
First, from \eqref{CPODE} we conveniently introduce the notation 
\begin{equation}\label{flow}
\frac{d \mathbf{X}_{t}^{\mathbf{u}}}{dt}= \mathbf{u}(t,\mathbf{X}^{\mathbf{u}}_{t}),
\end{equation}
and due to a simple calculation we obtain  
\begin{equation}
\label{flowU+U0}
\frac{d \mathbf{X}_{t}^{\mathbf{u}+ \bar{\mathbf{u}}}}{dt}= \mathbf{u}(t,\mathbf{X}_{t}^{\mathbf{u}+\bar{\mathbf{u}}}) +  \bar{\mathbf{u}}(t,\mathbf{X}_{t}^{\mathbf{u}+\bar{\mathbf{u}}}). 
\end{equation}
Then, we consider the following 
\begin{proposition} 
\label{FRECDERX}
The Fréchet derivative of the flow $\mathbf{X}_{t}^{\mathbf{u}+\bar{\mathbf{u}}}$, at $\mathbf{u}= \mathbf{u}_0$, is given for any time-constant vector field $\mathbf{w}(\cdot)$ by 
\begin{equation}
\label{FrecDervX}
\begin{aligned}
&D_{\mathbf{u}} \mathbf{X}_{t}^{\mathbf{u}+\bar{\mathbf{u}}}|_{\mathbf{u}= \mathbf{u}_0}[\mathbf{w}]
\\[5pt]
&= \exp\Big(\int_{0}^{t}  \big(\nabla_{x} \mathbf{u}_0(\mathbf{X}_{s}^{\mathbf{u}_0 + \bar{\mathbf{u}}}) 
+ \nabla_{x} \bar{\mathbf{u}}(s,\mathbf{X}_{s}^{\mathbf{u}_0+\bar{\mathbf{u}}})\big) \ ds \Big)  \int_0^t \mathbf{w}(\mathbf{X}_{s}^{\mathbf{u}_0 + \bar{\mathbf{u}}}) \, ds.
\end{aligned} 
\end{equation}
Analogously, the Fréchet derivative of its inverse, at $\mathbf{u}= \mathbf{u}_0$, is given by 
\begin{equation}
\label{FrecDervA}
\begin{aligned}
& D_{\mathbf{u}} \mathbf{A}_{t}^{\mathbf{u}+\bar{\mathbf{u}}}|_{\mathbf{u}= \mathbf{u}_0} [\mathbf{w}]= - \nabla_{x} \mathbf{A}_{t}^{\mathbf{u}_0 + \bar{\mathbf{u}}}(\mathbf{A}_t^{\mathbf{u}_0+\bar{\mathbf{u}}}) 
\\[5pt]
& \times  \exp \Big(\int_{0}^{t} \big(\nabla_{x} \mathbf{u}_0 (\mathbf{A}_{t-s}^{\mathbf{u}_0+\bar{\mathbf{u}}})+\nabla_{x} \bar{\mathbf{u}} (s,\mathbf{A}_{t-s}^{\mathbf{u}_0+\bar{\mathbf{u}}})\big) \ ds \Big)  
\int_0^t \mathbf{w}(\mathbf{A}_{t-s}^{\mathbf{u}_0 + \bar{\mathbf{u}}}) \, ds. 
\end{aligned}
\end{equation}
\end{proposition}

\begin{proof}
1. First, let us show \eqref{FrecDervX}.  For any $\delta \neq 0$ we obtain   
\[
\begin{aligned}
D_{\mathbf{u}} &\frac{d\mathbf{X}_{t}^{\mathbf{u}+\bar{\mathbf{u}}}}{dt} [\mathbf{w}]= \lim_{\delta\rightarrow 0} \frac{1}{\delta} 
\Big( \frac{d\mathbf{X}_{t}^{\mathbf{u}+ \delta \mathbf{w} + \bar{\mathbf{u}}}}{dt} - \frac{d\mathbf{X}_{t}^{\mathbf{u}+\bar{\mathbf{u}}}}{dt} \Big)
\\[5pt]
&= \lim_{\delta\rightarrow 0 } \Big( \frac{\mathbf{u}(t,\mathbf{X}_{t}^{\mathbf{u}+ \delta \mathbf{w} + \bar{\mathbf{u}}})- \mathbf{u}(t,\mathbf{X}_{t}^{\mathbf{u} + \bar{\mathbf{u}}})  }{\delta}   
\\[5pt]
&+  \frac{\bar{\mathbf{u}}(\mathbf{X}_{t}^{\mathbf{u}+ \delta \mathbf{w} + \bar{\mathbf{u}}})- \bar{\mathbf{u}}(\mathbf{X}_{t}^{\mathbf{u} + \bar{\mathbf{u}}})  }{\delta}  + \mathbf{w}(\mathbf{X}_{t}^{\mathbf{u}+ \delta \mathbf{w} + \bar{\mathbf{u}}}) \Big)
\\[5pt]
&= \big(\nabla_{x} \mathbf{u}(t,\mathbf{X}_{t}^{\mathbf{u} + \bar{\mathbf{u}}}) + \nabla_{x} \bar{\mathbf{u}}(\mathbf{X}_{t}^{\mathbf{u}+\bar{\mathbf{u}}})\big)  D_{\mathbf{u}} \mathbf{X}_{t}^{\mathbf{u} + \bar{\mathbf{u}}}[\mathbf{w}]
+  \mathbf{w}(\mathbf{X}_{t}^{\mathbf{u}+ \bar{\mathbf{u}}}).
\end{aligned}
\]
Then, we have
\[
\begin{aligned}
D_{\mathbf{u}}& \mathbf{X}_{t}^{\mathbf{u}+\bar{\mathbf{u}}}[\mathbf{w}]
\\[5pt]
&= \exp\Big(\int_{0}^{t}  \big(\nabla_{x} \mathbf{u}(s,\mathbf{X}_{s}^{\mathbf{u} + \bar{\mathbf{u}}}) 
+ \nabla_{x} \bar{\mathbf{u}}(\mathbf{X}_{s}^{\mathbf{u}+\bar{\mathbf{u}}})\big) \ ds \Big)  \int_0^t \mathbf{w}(\mathbf{X}_{s}^{\mathbf{u} + \bar{\mathbf{u}}}) \, ds, 
\end{aligned} 
\]
and taking $\mathbf{u}=\mathbf{u}_0$ we obtain
\[
\begin{aligned}
D_{\mathbf{u}}& \mathbf{X}_{t}^{\mathbf{u}+\bar{\mathbf{u}}}|_{\mathbf{u}= \mathbf{u}_0}[\mathbf{w}]
\\[5pt]
&= \exp\Big(\int_{0}^{t}  \big(\nabla_{x} \mathbf{u}_0(s,\mathbf{X}_{s}^{\mathbf{u}_0 + \bar{\mathbf{u}}}) 
+ \nabla_{x} \bar{\mathbf{u}}(\mathbf{X}_{s}^{\mathbf{u}_0+\bar{\mathbf{u}}})\big) \ ds \Big)  \int_0^t \mathbf{w}(\mathbf{X}_{s}^{\mathbf{u}_0 + \bar{\mathbf{u}}}) \, ds.
\end{aligned} 
\]

\medskip
2. Now, we prove \eqref{FrecDervA}. We claim that 
\begin{equation}
\label{formula}
   D_{\mathbf{u}} \mathbf{A}_{t}^{\mathbf{u}} [\mathbf{w}]= - \big( \nabla_{x} \mathbf{X}_{t}^{\mathbf{u}} \big)^{-1}   
   \,  D_{\mathbf{u}} \mathbf{X}_{t}^{\mathbf{u}} ( \mathbf{A}_{t}^{\mathbf{u}})  [\mathbf{w}].  
\end{equation} 
Indeed, we first observe that 
\[
\mathbf{X}_{t}^{\mathbf{u}}(\mathbf{A}_{t}^{\mathbf{u}})=Id \quad \text{and similarly,} \quad \mathbf{X}_{t}^{\mathbf{u}+\delta \mathbf{w}}(\mathbf{A}_{t}^{\mathbf{u}+\delta \mathbf{w} })= Id,
\]
for any real number $\delta$.
Then, we have for $\delta \neq 0$, 
\[
\begin{aligned}
0&= \frac{\mathbf{X}_{t}^{\mathbf{u}+\delta \mathbf{w}}(\mathbf{A}_{t}^{\mathbf{u}+ \delta \mathbf{w}}) - \mathbf{X}_{t}^{\mathbf{u}}(\mathbf{A}_{t}^{\mathbf{u}})  }{ \delta}
\\[5pt]
&= \frac{\mathbf{X}_{t}^{\mathbf{u}+\delta \mathbf{w}}(\mathbf{A}_{t}^{\mathbf{u}+ \delta \mathbf{w}})-\mathbf{X}_{t}^{\mathbf{u}
+\delta \mathbf{w}}(\mathbf{A}_{t}^{\mathbf{u}}) }{ \delta} + \frac{\mathbf{X}_{t}^{\mathbf{u}+\delta \mathbf{w}}(\mathbf{A}_{t}^{\mathbf{u}})- \mathbf{X}_{t}^{\mathbf{u}}(\mathbf{A}_{t}^{\mathbf{u}})}{\delta}
\\[5pt]
&= I_{1} + I_{2} 
\end{aligned}
\]
with obvious notation. We compute, 
\[
\begin{aligned}
\lim_{\delta \rightarrow 0}  I_{1}&= \lim_{\delta \rightarrow 0}  \frac{ \mathbf{X}_{t}^{\mathbf{u}+\delta \mathbf{w}}(\mathbf{A}_{t}^{\mathbf{u}+ \delta \mathbf{w}})-\mathbf{X}_{t}^{\mathbf{u}+\delta \mathbf{w}}(\mathbf{A}_{t}^{\mathbf{u}})  }{ \delta}
\\[5pt]
&= \nabla_{x} \mathbf{X}_{t}^{\mathbf{u}}(\mathbf{A}^{\mathbf{u}}_{t}) \  D_{\mathbf{u}} \mathbf{A}_{t}^{\mathbf{u}}[\mathbf{w}],  
\end{aligned}
\]
and 
\[
\lim_{\delta \rightarrow 0}  I_{2} = (D_{\mathbf{u}} \mathbf{X}_{t}^{\mathbf{u}})(\mathbf{A}^{\mathbf{u}}_{t}) [\mathbf{w}]. 
\]
Hence we deduce 
$$
 D_{\mathbf{u}} \mathbf{A}_{t}^{\mathbf{u}} [\mathbf{w}]= - \nabla_{x} \mathbf{A}_{t}^{\mathbf{u}}  
   \,  D_{\mathbf{u}} \mathbf{X}_{t}^{\mathbf{u}} ( \mathbf{A}_{t}^{\mathbf{u}})  [\mathbf{w}].  
$$
Therefore, due to \eqref{formula} we obtain  
\[
D_{\mathbf{u}} \mathbf{A}_{t}^{\mathbf{u}+\bar{\mathbf{u}}} = -  \nabla_{x} \mathbf{A}_{t}^{\mathbf{u}+\bar{\mathbf{u}}}  
D_{\mathbf{u}} \mathbf{X}_{t}^{\mathbf{u}+\bar{\mathbf{u}}}( \mathbf{A}_{t}^{\mathbf{u}+\bar{\mathbf{u}}}), 
\]
and consequently we have 
$$
\begin{aligned}
&D_{\mathbf{u}} \mathbf{A}_{t}^{\mathbf{u}+\bar{\mathbf{u}}}|_{\mathbf{u}=\mathbf{u}_0} [\mathbf{w}]= - \nabla_{x} \mathbf{A}_{t}^{\mathbf{u}_0+\bar{\mathbf{u}}}(\mathbf{A}_t^{\mathbf{u}_0+\bar{\mathbf{u}}})  
\\[5pt]
& \times  \exp \Big(\int_{0}^{t} \big(\nabla_{x}\mathbf{u}_0 (\mathbf{A}_{t-s}^{\mathbf{u}_0+\bar{\mathbf{u}}}) + \nabla_{x} \bar{\mathbf{u}} (s,\mathbf{A}_{t-s}^{\mathbf{u}_0+\bar{\mathbf{u}}})\big) \ ds \Big)  
\int_0^t \mathbf{w}(\mathbf{A}_{t-s}^{\mathbf{u}_0+ \bar{\mathbf{u}}}) \, ds. 
\end{aligned}
$$
\end{proof}

\section*{Data availability statement}

Data sharing is not applicable to this article as no data sets 
were generated or analysed during the current study.

 \section*{Acknowledgements}

The author Wladimir Neves has received research grants from CNPq
through the grant  313005/2023-0, 403675/2025-1, and also by FAPERJ 
(Cientista do Nosso Estado) through the grant E-26/204.171/2024. 

\medskip
The author Christian Olivera is partially supported by FAPESP by the grant  $2020/04426-6$,  
 by FAPESP-ANR by the grant Stochastic and Deterministic Analysis for Irregular Models$-2022/03379-0$ 
 and CNPq by the grant $422145/2023-8$.

 \section*{Conflict of Interest}
There is no conflict of interests.


\end{document}